\documentclass[reqno,12pt]{amsart}
\usepackage[english]{babel}
\usepackage{amsmath,amsfonts,amsthm,amssymb}
\usepackage{graphicx}
\usepackage{enumitem}
\setlist[itemize]{%
	topsep=2pt,
	leftmargin=0.6cm
}

\usepackage{hyperref}
\usepackage{mdframed}
\usepackage[svgnames]{xcolor}

\hypersetup{
	allbordercolors=.,
	colorlinks=true,
	citecolor=NavyBlue,
	linkcolor=DarkRed,
	urlcolor=NavyBlue,
	bookmarksopen=true,
	pdfdisplaydoctitle=true
}

\graphicspath{{./pics/}}
\usepackage[a4paper,margin=3.00truecm]{geometry}
\usepackage{bbm}
\usepackage{soul}

\DeclareMathOperator{\spt}{spt}

\DeclareMathOperator{\tr}{tr}
\DeclareMathOperator{\Sym}{Sym}
\DeclareMathOperator{\adj}{adj}
\DeclareMathOperator{\Proj}{Proj}
\DeclareMathOperator{\reff}{ref}
\DeclareMathOperator{\addd}{add}
\DeclareMathOperator{\Per}{Per}

\def\R{\mathbb{R}}

\def\F{\mathcal{F}}
\def\L{\mathcal{L}}

\newcommand{\be}{\begin{equation}}
\newcommand{\ee}{\end{equation}}

\numberwithin{equation}{section}
\theoremstyle{plain}

\newtheorem{theo}{Theorem}[section]

\newtheorem{prop}[theo]{Proposition}

\theoremstyle{definition}
\newtheorem{rema}[theo]{Remark}

\newmdtheoremenv[
	linewidth=2pt,
	roundcorner=5pt,
	skipabove=5pt,
	innertopmargin=0,
	splittopskip=\topskip,
	ntheorem=true,
	innerbottommargin=\topskip,
]{algo}{Algorithm}

\title[Numerical approximation of Blaschke-Santal\'o diagrams]{On the numerical approximation of Blaschke-Santal\'o diagrams using Centroidal Voronoi Tessellations}

\author[B. Bogosel]{Beniamin Bogosel}
\author[G. Buttazzo]{Giuseppe Buttazzo}
\author[E. Oudet]{Edouard Oudet}

\date{}

\begin{document}

\begin{abstract}
Identifying Blaschke-Santal\'o diagrams is an important topic that essentially consists in determining the image $Y=F(X)$ of a map $F:X\to\R^d$, where the dimension of the source space $X$ is much larger than the one of the target space. In some cases, that occur for instance in shape optimization problems, $X$ can even be a subset of an infinite-dimensional space. The usual Monte Carlo method, consisting in randomly choosing a number $N$ of points $x_1,\dots,x_N$ in $X$ and plotting them in the target space $\R^d$, produces in many cases areas in $Y$ of very high and very low concentration leading to a rather rough numerical identification of the image set. On the contrary, our goal is to choose the points $x_i$ in an appropriate way that produces a uniform distribution in the target space. In this way we may obtain a good representation of the image set $Y$ by a relatively small number $N$ of samples which is very useful when the dimension of the source space $X$ is large (or even infinite) and the evaluation of $F(x_i)$ is costly. Our method consists in a suitable use of {\it Centroidal Voronoi Tessellations} which provides efficient numerical results. Simulations for two and three dimensional examples are shown in the paper.
\end{abstract}

\maketitle


 
%%%%%%%%%%%%%%%%%%%%%%%%%%%%%%%%%%%%%%%%%%%%%%%%%%
\section{Introduction}\label{sintro} 

In several problems one is faced with the question of identifying the image of a map $F:X\to\R^d$, given $X$ the set of admissible choices and $F$ a given map which describes the performances $F(x)$ of the given choice $x\in X$. In this article we study the approximation of $F(X)$ when the dimension $d$ is small, which amounts to say that the performances $F(x)$ can be summarized by a few outputs, with respect to the dimension of $X$ which may even be infinite. The subset $F(X)\subset\R^d$ is often called Blaschke-Santal\'o diagram when $X$ is a class of {\it shapes}, and may involve quantities such as the volume, the perimeter, the torsional rigidity, the eigenvalues of the Laplace operator $-\Delta$, and other similar geometrical or analytical quantities (see for instance \cite{bucur1999attainable, missingADR, luzu22}).
 
Even when the set $X$ is a subset of a finite dimensional Euclidean space $\R^N$, the identification of the image $F(X)$ through a numerical procedure may present deep and unexpected difficulties. The naive approach consisting in generating a random uniform sampling of $X$ by means of a discrete set $\{x_i\}_{i\in I}$, associated to the corresponding outputs $\{F(x_i)\}_{i\in I}$, may not produce a satisfactory approximation. Some parts of the image $F(X)$ may be very rarely explored by a uniform random sampling of $X$. In these cases one needs a very large number of samples in order to have a quite accurate description of the set $F(X)$. This phenomenon happens to have a dramatic impact when the evaluation of $F$ on the chosen sample requires the solution of one or more partial differential equations. In shape optimization, for instance, this procedure may be too expensive in terms of computational time. Such contexts require to develop a new approach to get a precise description of the image $F(X)$ using a relatively small number of sampling points. The choice of the sampling has to be adjusted carefully in order to comply with the complexity of the map $F$.
 
In the present paper we develop a new method based on Voronoi tessellations which seems much more efficient than the standard random uniform approach. We describe in the following sections the numerical method, and we show some algebraic examples in which the efficiency of our procedure is clearly outlined.

The study of the range of scale invariant ratios between geometric quantities was initiated by Santal\'o in \cite{Santalo} and Blaschke in \cite{Blaschke}. If the geometric image of scale invariant ratios is completely characterized, then all possible inequalities between these quantities are known. In practice, often three geometric functionnals are used to generate at least two scale invariant ratios. 

This approach has been investigated recently in a shape optimization context. We mention \cite{missingADR} concerning the diagram given by the area, the diameter and the inradius. In \cite{ftouhi-lamboley} the inequalities between volume, perimeter and the first Dirichlet-Laplace eigenvalue are investigated. The Cheeger inequality was investigated in \cite{ftouhi-Cheeger} and inequalities involving the first Dirichlet eigenvalue, the torsion and the volume are studied in \cite{ftouhi-polya}. 

In most situations, a complete analytical understanding of the resulting Blaschke-Santal\'o diagrams is not available. This motivates the use of numerical tools. A first method is generating random shapes and computing quantities of interest. This method is illustrated in some works cited above. A more rigorous numerical approach is solving numerical optimization problems finding extreme points for vertical or horizontal slices of the diagram. This method is used in \cite{ftouhi-numerics} using methods described in \cite{AntunesBogosel} and \cite{DiscreteConvex} to perform numerical shape optimization among convex sets. 

This article proposes a completely new alternative approach which generates uniformly distributed samples in the Blaschke-Santal\'o diagram. Implicitly, our method also provides boundary points for the diagram. Compared to \cite{ftouhi-numerics} where multiple constrained numerical optimizations are solved, we use a global iterative process, and we solve numerically a global optimization problem providing a geometrical description of the diagram. We illustrate the method for an algebraic example involving the trace and determinant of symmetric matrices with entries in $[-1,1]$. This simple example is already non-trivial starting from $3\times 3$ matrices. In a second stage we investigate numerically the diagram associated to the area, perimeter and moment of inertia among convex shapes with two axes of symmetry. 

%%%%%%%%%%%%%%%%%%%%%%%%%%%%%%%%%%%%%%%%%%%%%%%%%%
\section{Approximation framework}\label{sec:framework}

\subsection{Optimal transport framework}\label{sec:optimal-transport}

Consider a continuous map $F:X\to\R^d$, with $X$ a compact metric space. In order to have a careful description of the image set $F(X)$ we could randomly choose some points $\{x_1,\dots,x_n\}$ in $X$ and, for a large $n$, the set $\{F(x_k)\ :\ k=1,\dots,n\}$ would give an approximate description of the full image set $F(X)$. However, even if $X$ is a subset of an Euclidean finite dimensional space, due to the nonlinearity of the map $F$, the number $n$ that is necessary to have a rather accurate description of the image set $F(X)$ could be extremely high. In other words, a uniform random choice of points $\{x_1,\dots,x_n\}$ in $X$ does not produce in general a well distributed sequence $F(x_1),\dots,F(x_n)$, and concentration/rarefaction effects very often occur. We should then make the random choice of the points $\{x_1,\dots,x_n\}$ in $X$ according to a probability measure that is not uniform and that depends on the function $F$, in order to obtain a well distributed sequence $F(x_1),\dots,F(x_n)$. If $\L$ is the Lebesgue measure on $F(X)$ the probability measure $\mu$ governing the random choice of points on $X$ should then be such that $F^\#\mu=\L$, being $F^\#$ the push-forward operator related to the function $F$, verifying 
\[ \mathcal L(B) =  F^\#\mu(B) = \mu (f^{-1}(B)),\]
for every measurable set $B\subset F(X)$.

\begin{theo}\label{exist}
Let $X,Y$ be compact metric spaces and let $F:X\to Y$ be a continuous function, with $Y=F(X)$. Then, for every probability measure $\nu$ on $Y$ there exists a probability measure $\mu$ on $X$ such that
$$F^\#\mu=\nu.$$
\end{theo}

\begin{proof}
Let $\nu$ be a probability measure on $Y$ and let $(\nu_n)$ be a sequence of discrete probability measures on $Y$ with $\nu_n\rightharpoonup\nu$ weakly*. Each $\nu_n$ has the form
$$\nu_n=\frac1n\sum_{k=1}^n\delta_{y_{n,k}}$$
where $y_{n,k}$ are suitable points in $Y$. Since $Y=F(X)$ we may take $x_{n,k}\in X$ such that $F(x_{n,k})=y_{n,k}$ and define a discrete probability measure $\mu_n$ on $X$ by
$$\mu_n=\frac1n\sum_{k=1}^n\delta_{x_{n,k}}.$$
Then $F^\#\mu_n=\nu_n$ and, possibly passing to a subsequence, we may assume that $\mu_n\rightharpoonup\mu$ weakly* for some probability measure $\mu$ on $X$. Passing now to the limit as $n\to\infty$, we obtain $F^\#\mu=\nu$.
\end{proof}

The usual uniform Monte Carlo method consists in taking $\mu$ as the Lebesgue measure on the source space $X$. 
%however, in this way in the target space $Y$ some areas of very high and very low concentration may appear, and 
In this case it often happens that the image points are unevenly distributed in $Y$, making the numerical identification of the image set $Y$ of a rather poor quality. On the contrary, our goal is to construct (a discrete approximation of) a measure $\mu$ in order to obtain a well distributed image measure $\nu$, as close as possible to the Lebesgue measure on $Y$. The proof of Theorem \ref{exist} is clearly non-constructive, since the image set $Y=F(X)$ is not a priori known. We then need a constructive method that provides the probability measure $\mu$ in the theorem above through an approximation procedure. This is the goal of next sections.

To further motivate our approach, let us recall the classical optimal transport problem raised by Monge \cite{monge}. Given probability measures $\eta, \nu$ on $Y$ and a cost function $c : Y \times Y \to [0,\infty]$ solve 
\begin{equation}
\inf_{T : Y\to Y} \left\{ \int_Y c(y,T(y)) d\mu(y) : T^\#\eta = \nu \right\}.
\label{eq:ot-monge}
\end{equation}
In the case the cost is simply given by the Euclidean distance squared $c(y_1,y_2)=\|y_1-y_2\|^2$, the infinimum in \eqref{eq:ot-monge} is called the Wasserstein 2-distance and is denoted $W_2^2(\eta,\nu)$. 

Our objective is to approximate the Lebesgue mesure in the image $Y=F(X)$ with a discrete set of points. In the case where the measure $\eta$ is discrete, given by a sum of Dirac masses $\eta_M = \sum_{i=1}^M a_i \delta_{y_i}$ then it is known that if $\eta_M$ solves the problem
\begin{equation}
\min \{W_2(\eta_M,\nu) : \#(\spt(\eta_M))\leq M\}
\label{eq:location-problem}
\end{equation}
the so-called \emph{location problem}, then the points $(y_i)_{i=1,...,M}$ correspond to a Centroidal Voronoi Tessellation on $Y$. For more details see \cite[Section 6.4.1, Box 6.6, Exercise 39]{Santambrogio}.

This motivates our approach, detailed in the following sections: Find sample points $(x_i)_{i=1,...,M}$ in $X$ such that their images $(F(x_i))_{i=1}^M \subset Y$ give the best representation of the Lebesgue measure on $F(Y)$ in the sense of \eqref{eq:location-problem}, i.e. $(F(x_i))_{i=1,...,M}$ form a Centroidal Voronoi Tesellation in the image $Y$.

%\Red{Can you put here some examples, like the ones you discussed in the talk at Orsay:
%\begin{itemize}
%	\item if $\mu$ is the Lebesgue measure on $X$ then $\nu$ is not uniform
%	\item Interpretations optimal transport for Voronoi (Wasserstein-2?) and maximizing the minimal distance between images (other Wasserstein?)
%\end{itemize}}

\subsection{Centroidal Voronoi Tessellations}\label{sec:voronoi-algs}

Consider $D$ a compact connected region in $\Bbb{R}^d$ (typically a rectangular box). Given $M$ points $y_i\in D$, consider the associated Voronoi diagram consisting of a partition $(V_i)_{i=1}^M$ of $D$ such that
\begin{equation}\label{eq:voronoi}
V_i=\big\{y\in D\ :\ |y-y_i|\le|y-y_j|\text{ for every }1\le j\le M,\ j\ne i\big\}.
\end{equation}
In other words, $V_i$ contains all points in $D$ which are closer to $y_i$ compared to the remaining images $(y_j)_{j\ne i}$. Note that in our convention, Voronoi cells are bounded and are subsets of $D$. The Voronoi cells are, in general, different in volume and are not necessarily uniform, for a general distribution of points. See the example in Figure \ref{fig:voronoi-example} where a Voronoi diagram corresponding to $10$ random points in the square $[-1,1]^2$ is shown. The Voronoi points are represented with red squares and the centroids of the Voronoi cells are represented with blue points.
\begin{figure}
	\centering 
	\includegraphics[height=0.4\textwidth]{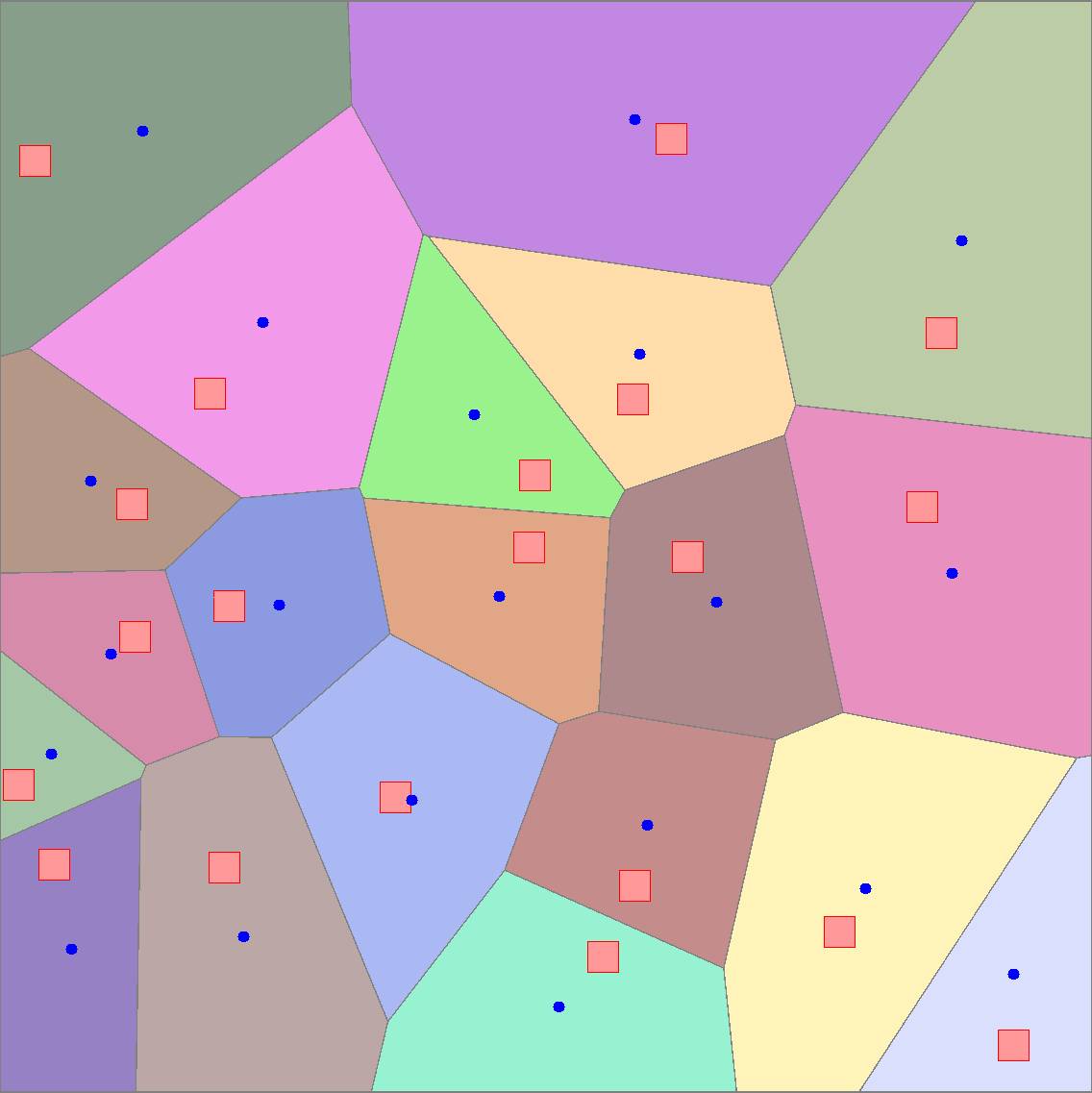}\qquad
	\includegraphics[height=0.4\textwidth]{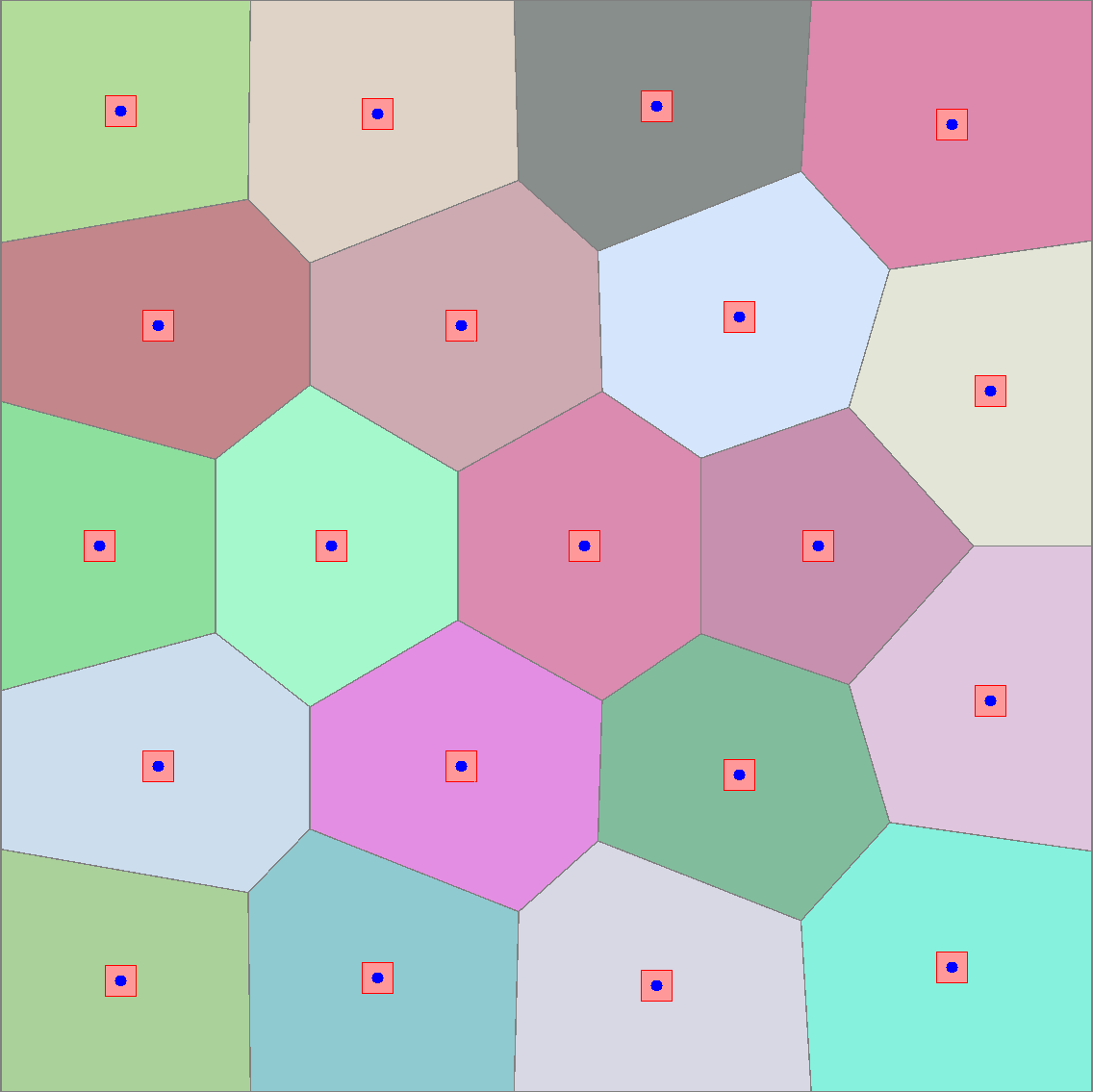}
	\caption{General Voronoi diagram (left) vs Centroidal Voronoi Diagram (right). The points $(y_i)$ are depicted by squares, the centroids of the Voronoi regions are depicted by dots.}
	\label{fig:voronoi-example}
\end{figure}

There exist, however, particular classes of Voronoi diagrams which have cells that are more uniform in size, called Centroidal Voronoi Tessellations (CVT). For such diagrams, the point $y_i$ which determines the Voronoi cell $V_i$ coincides with the centroid of the region $V_i$. An example of CVT is shown in Figure \ref{fig:voronoi-example} where it can be observed that the Voronoi points overlap with the centroids of the associated cells. The Voronoi regions for such a configuration have uniform sizes. Relevant usages of CVTs involve optimal quantization, data compression, optimal quadrature and mesh generation. We refer to \cite{du_emelianenko_acceleration, Du1999, cvt-levy, lloyd} for more details and references. There exist iterative algorithms that produce CVTs starting from general Voronoi diagram. the most basic one being Lloyd's algorithm, described as follows.

{\bf Lloyd's Algorithm.} For a prescribed number of iterations $M \geq 1$, successively replace the Voronoi points $y_i$ with the centroids $c_i$ of their corresponding Voronoi cells $V_i$. The algorithm was introduced in \cite{lloyd} and has a straightforward implementation. Convergence properties for Lloyd's algorithm are investigated in \cite{SabinGray}, \cite{Emelianenko2008}. While being easy to implement, this algorithm is not the most efficient from a practical point of view for reasons expressed in the following. 

In \cite{Du1999}, \cite{cvt-levy} it is shown that the Centroidal Voronoi Tessellation are critical points for the energy
\begin{equation}\label{eq:cvt-energy}
 G(y_1,..,y_M) = \sum_{i=1}^M \int_{V_i} |x-y_i|^2 dx.
\end{equation}
Indeed, an immediate computation which can be found, for example, in \cite{Du1999} shows that
\begin{equation}\label{eq:deriv-cvt}
\frac{\partial G}{\partial y_i}(y_1,...,y_M) = 2|V_i|(y_i-c_i), 1 \leq i \leq M
\end{equation}
where $c_i$ are the centroids of $V_i$, for $1 \leq i \leq M$. For a more direct identification of the gradient one may consider the following alternative formula:
\[\begin{split}
G(y_1,..,y_M) = &\int_D |x|^2\,dx+\sum_{i=1}^M\left(-2 y_i\cdot\int_{V_i} x\,dx + |y_i|^2 |V_i|\right)\notag\\
=&\int_D |x|^2\,dx+\sum_{i=1}^M |V_i|(|y_i|^2-2y_i\cdot c_i).
\end{split}\]
Notice that the quadratic term is constant and can be ignored, leading to a simplified energy function.

In view of the gradient formula \eqref{eq:deriv-cvt}, Lloyd's algorithm can be viewed as a gradient descent algorithm for minimising the energy \eqref{eq:cvt-energy} with no need for a step size control \cite{Du1999}. Acceleration techniques for improving the convergence of Lloyd's algorithm are considered in \cite{du_emelianenko_acceleration}, where a Newton algorithm is discussed. This proves efficient for a small number of Voronoi cells, but is computationally heavy for large $M$.

It is generally agreed in the numerical optimization practice that gradient descent algorithms have slow convergence for ill-conditioned problems \cite{nocedal-wright}. For large scale problems, quasi-Newton methods like the low memory BFGS algorithms (lbfgs) \cite[Chapter 7]{nocedal-wright} have better convergence properties. In \cite{cvt-levy} such a quasi-Newton method was used for minimizing \eqref{eq:cvt-energy}, providing a faster method for computing CVTs. It is well known that CVTs are not necessarily unique, but on the other hand CVTs obtained through a variational process have improved stability properties as already underlined in \cite{cvt-levy}. Thus an alternative way of finding CVTs is the following.

{\bf Variational CVT.} Minimize numerically the energy \eqref{eq:cvt-energy} using a quasi-Newton algorithm. Benchmarks presented in \cite{cvt-levy} show that such an algorithm is more efficient than Lloyd's Algorithm.

\subsection{Approximation of Blaschke-Santal\'o diagrams using CVTs}

The objective of this article is to provide numerical tools which allow the approximation of Blaschke-Santal\'o diagrams, i.e. the image of a mapping 
\begin{equation}
\label{eq:Fmap}
F: X \to \Bbb{R}^d
\end{equation}
where $X\subset \Bbb{R}^N$ is the set of admissible parameters, containing upper and lower bounds and other eventual constraints. 

Given $M$ a positive integer, consider a random set of samples $x_1,\dots,x_M\in X\subset\R^N$. As underlined in the introduction, the images $F(x_1),\dots,F(x_M)$ are not necessarily uniformly distributed in the image $F(X)\subset\R^d$. Our goal is to find a choice of the points $x_1,\dots,x_M$ in such a way that their images $F(x_1),\dots,F(x_N)$ are uniformly distibuted inside the image set $F(X)$.

In the following, we assume $F(X)$ is bounded and consider $D$ a bounding box containing $F(X)$ strictly in its interior. Denote by $y_i=F(x_i)$ the images for the initial sampling. We obviously have $y_i\in D$. Consider now the Voronoi diagram associated to the points $y_i$ defined by \eqref{eq:voronoi}. Since our goal is to obtain a more uniform distribution of the images, we search points $y_i=F(x_i)$ which produce a Voronoi diagram that is as close as possible to a CVT. Inspired from the results recalled in Section \ref{sec:voronoi-algs} we propose two algorithms for approximating Blaschke-Santal\'o diagrams.

{\bf Lloyd algorithm with projection.} Lloyd's algorithm is simple to implement, for general CVTs, simply replacing the points with the corresponding centroids. When dealing with BS diagrams one would like, for each sample $x_i$, $i=1,...,M$ to replace it with another admissible sample $\bar x_i$ such that $F(\bar x_i)$ is the centroid of the Voronoi region associated to $F(x_i)$. We are faced with two issues:
\begin{enumerate}[label=(\alph*)]
	\item Given a point $c$ in the image $F(X)$, find a sample realizing $c$, i.e. find $x\in X$ such that $F(x) = c$.
	\item Given a general point $c\in \Bbb{R}^d$ which may not be in the image $F(X)$, find $x \in X$ such that $F(x)$ is \emph{closest} to $c$ in a sense to be defined.
\end{enumerate}
Both aspects enumerated above can be covered using a single optimization problem:
\begin{equation}\label{eq:inverse-sample}
\min_{ x \in X} \|F(x)-c\|.
\end{equation}
In cases where $X$ is a compact set and $F$ is at least of class $C^1$ problem \eqref{eq:inverse-sample} admits solutions and efficient approximations can be found using standard numerical optimization algorithms. 

We are thus lead to the following natural algorithm.

\begin{algo}[Lloyd Blaschke-Santal\'o]
	
	Input: number of samples $M$, number of iterations $q$, choose a bounding box $D$ for the image $F(X)$, tolerance $\varepsilon>0$. 
	
	Initialization: generate $M$ random samples $x_i \in X$, $i=1,...,M$
	
	Loop: For each one of the $q$ iterations do:
	\begin{itemize}
		\item Compute $y_i = F(x_i)$
		\item Compute the Restricted Voronoi Diagram $(V_i)_{i=1}^M$ associated to the points $y_i$. Compute the centroids $c_i$ of the regions $V_i$, $i=1,...,M$.
		\item For each one of the $c_i$ solve \eqref{eq:inverse-sample} and replace $x_i$ with the numerical solution $\bar x_i$.
		\item If for every $i=1,...,M$ we have $\|F(x_i)-F(\bar x_i)\|<\varepsilon$ stop.
	\end{itemize}
\label{algo:lloyd}
\end{algo}

%\ref{algp:lloyd}
In our implementation and in the sequel, the norm considered in problem \eqref{eq:inverse-sample} is the Euclidean one. However other choices are possible and may give different behavior for the algorithm. Supposing Algorithm \ref{algo:lloyd} converges for a given threshold $\varepsilon$ we obtain a configuration of samples $(x_i)_{i=1}^M$ such that:
\begin{itemize}
	\item Whenever $x_i$ is such that the centroid $c_i$ of $V_i$ belongs to $F(X)$ we have $\|F(x_i)-c_i\|<\varepsilon$.
	\item Whenever $x_i$ is such that the centroid $c_i$ of $V_i$ does not belong to $F(X)$ we have $\|\Proj_{F(X)}(c_i) -F(x_i)\|<\varepsilon$. We used the classical notation for the projection operator $\Proj_S(y) = \{z \in S : \|z-y\| \text{ is minimal}\}$. In particular, $F(x_i)$ will be a boundary point for $F(X)$.
\end{itemize}

The drawbacks of Algorithm \ref{algo:lloyd} are similar to the ones of Lloyd's algorithm compared to the variational CVT. One may interpret Algorithm \ref{algo:lloyd} as a fixed point or gradient descent algorithm with projection. Such algorithms can be improved using quasi-Newton methods as described in the following.

{\bf Variational CVT for Blasche-Santal\'o diagrams}. Similar to \cite{cvt-levy} we formulate a minimization problem. We propose to minimize the composition of \eqref{eq:cvt-energy} with the parametrization \eqref{eq:Fmap}. For a given number of samples $M$ we consider the functional $H : X^M \subset (\Bbb{R}^N)^M \to \Bbb{R}$ given by
\begin{equation}\label{eq:functional-composition}
H(x_1,\dots,x_M) = G(F(x_1),F(x_2),\dots,F(x_M)),
\end{equation}
with $G$ defined in \eqref{eq:cvt-energy}. In practice, we minimize $H$ using quasi-Newton methods with eventual bound and linear constraints characterizing the parameter set $X$. This type of problems can easily be handled using available implementations (\texttt{fmincon} in Matlab, Knitro see \cite{byrd2006k}). Assuming the function $F$ defined in \eqref{eq:Fmap} is differentiable, the derivatives of $H$ can be expressed with
\begin{equation}\label{eq:gradH}
\frac{\partial H}{\partial x_i} = DF^T(x_i) \frac{\partial G}{\partial y_i}(F(x_1),\dots,F(x_N))=2|V_i| DF^T(x_i)(F(x_i)-c_i),
\end{equation}
where $DF(x_i) \in \Bbb{R}^{d\times n}$ is the Jacobian of $F$ evaluated at $x_i$ and $c_i$ is the centroid of the Voronoi region $V_i$, $i=1,...,M$. We arrive at the following algorithm.

\begin{algo}[Variational CVT Blaschke-Santal\'o]	\label{algo:var-cvt}
	
	Input: number of samples $M$, number of iterations $q$, choose a bounding box $D$ for the image $F(X)$, tolerance $\varepsilon>0$. 
	
	Initialization: generate $M$ random samples $x_i \in X$, $i=1,...,M$
	
	Optimization: minimize \eqref{eq:functional-composition} using gradient information given by \eqref{eq:gradH}:
	\[ \min_{(x_i) \in X^M} H(x_1,...,x_M).\]
\end{algo}

Minimizing $H$ on $X^M$ will produce images $F(x_i)$, $i=1,\dots,M$ that are equidistributed in $F(X)$ in the following sense.

\begin{prop}Assume the parameter set $X\subset \Bbb{R}^N$ is compact and $F$ is $C^1$. Suppose $x_1^*,\dots,x_M^*$ minimizes \eqref{eq:functional-composition} on $X^M$ and denote by $y_i^* = F(x_i^*)$, $i=1,\dots,M$ the corresponding images. If $x_i^*$ is an interior point of $X$ and $DF(x_i^*)$ is of full rank, then $y_i^*$ is the centroid of the Voronoi cell associated to $x_i^*$.

%(ii) If $x_i^*$ is an interior point of $X$ and $DF(x_i^*)$ is not of full rank, then $F(x_i^*)-c_i \in \ker DF(x_i^*)^T$. Then $F(x_i^*)$ is a boundary point for $F(X)$ and $F(x_i^*)-c_i$ is a normal vector for $\partial F(X)$.
\end{prop}

\begin{proof}
If $x_i^*$ is an interior point for $X$ then $\frac{\partial H}{\partial x_i} = 0$. In view of \eqref{eq:gradH}, if $DF(x_i^*)$ is of full rank, then $F(x_i^*)=c_i$, i.e. $F(x_i^*)$ is the centroid of the region $V_i$.
\label{prop:optim}
\end{proof}

\begin{rema}
In practice, if Proposition \ref{prop:optim} does not apply, we may have the following situations.
\begin{itemize}
\item[(a)]If $F(x_i)$ is a boundary point for $F(X)$ then $F(x_i)$ is not necessarily equal to the centroid $c_i$ of the region $V_i$.
\item[(b)]Interior points of $F(X)$ for which the Jacobian $DF$ is not of full rank may act as boundary points. We observe this behavior in the numerical simulations.
\end{itemize}
\end{rema}

The minimization of the functional \eqref{eq:functional-composition} is straightforward if the Voronoi diagram associated to a set of points can be computed. We use the routine \texttt{compute\_RVD} developed following the results in \cite{cvt-levy} from the library Geogram.
\begin{center}
\href{https://github.com/BrunoLevy/geogram}{\nolinkurl{https://github.com/BrunoLevy/geogram}}
\end{center}
The optimization is performed in Matlab/Julia using \texttt{fmincon} or the Artelys Knitro software in Algorithm \ref{algo:var-cvt}. Details regarding the optimization procedure and more specific aspects regarding the problem at hand are shown in the next section.

%%%%%%%%%%%%%%%%%%%%%%%%%%%%%%%%%%%%%%%%%%%%%%%%%%
\section{Application I: algebraic functions}\label{sec:matrix}

We start with an algebraic example which is easy to state, but quickly becomes challenging. For $d\ge2$ consider the space $\Sym_d([-1,1])$ of symmetric $d\times d$ matrices with real entries in the interval $[-1,1]$. We apply our algorithm to the study of the $(\tr,\det)$ diagram ($\tr$ denotes the trace, $\det$ denotes the determinant). More precisely, consider the application $F:\Sym_d([-1,1]) \to \Bbb{R}^2$ defined by
\begin{equation}
 F(A) = (\tr(A),\det(A)).
 \label{eq:trace-det}
\end{equation}
Our goal is to identify the image $J(\Sym_d([-1,1]))$ for some particular choice of $d$.

While for $d=2$ a complete analytical description of the diagram is possible, for $d\geq 3$ the problem becomes challenging. On the other hand, the numerical method we propose is efficient and shows a clear description of the corresponding diagram.

We represent symmetric matrices of size $d\times d$ as a vector in $\Bbb{R}^{d(d+1)/2}$, the concatenation of the diagonals $j-i=0,1,...,d-1$. With this convention the gradient of the trace is equal to
\[ \nabla \tr(A) = (1,1,...,1,0,0...0),\]
where the first $d$ elements are zero. Therefore the jacobian matrix $DF(A)$ has rank at least $1$. Partial derivatives of the determinant with respect to the entries of the matrix are components of the adjugate matrix $\adj(A)$. The elements of $\adj(A)$ on position $(j,i)$ are equal to $(-1)^{i+j}$ times the $(i,j)$ minor of $A$, the determinant of the matrix obtained from $A$ when removing the $i$-th line and $j$-th column. In particular $A \cdot \adj(A) = \det(A) I$. 

As a consequence, the Jacobian $DF(A)$ has rank one if and only if the matrix $\adj(A)$ is a multiple of the identity. Then $\det(A)=0$ or $A$ is also a multiple of the identity. In particular, the Jacobian matrix $DF(A)$ is of rank $1$ for all diagonal matrices. Next, we consider in detail the case $d=2$ for which an analytic description of the $(\tr,\det)$ diagram is available.

{\bf Analysis of the two dimensional case.} For $d=2$ we obviously have $\tr(A) \in [-2,2]$ with extremal values attained when diagonal elements are all equal to $\pm 1$. This shows that the diagram is contained in $[-2,2]\times \Bbb{R}$.

Consider $A =\begin{pmatrix}
a&c \\ c& b
\end{pmatrix}$ and fix $q = a+b$. Then
$$\det(A) = ab-c^2\leq ab \leq \frac{1}{4}(a+b)^2 \leq q^2/4.$$
Thus, the upper part of the diagram is the curve $[-2,2]\ni q \mapsto q^2/4$.

We also have $ab = a(q-a)$ which is a concave function on $[\max\{q-1,-1\},\min\{q+1,1\}]$. The minimum is attained at one of the endpoints of the interval. Investigating this minimum with respect to $q \in [-2,2]$ we find that the lower bound of the diagram is given by $q\mapsto |q|-2$. The Jacobian of $A$ with respect to variables $a,b,c$ is singular if and only if it is diagonal, corresponding to the upper bound $q \mapsto q^2/4$.

{\bf Illustration of the numerical algorithms.} In the following we apply the algorithms proposed in Section \ref{sec:framework} to study the proposed diagram. The bounding boxes $D$ are considered as follows:
\begin{itemize}
\item $d=2$: $D= [-2.5,2.5] \times [-2.5,2.5]$
\item $d=3$: $D= [-5,5]\times [-5.5]$
\item $d=4$: $D= [-6,6]\times [-20,20]$.
\end{itemize}
We apply Algorihm \ref{algo:lloyd} for $M=200$ samples, with a maximum number of $q=1000$ iterations and a tolerance $\varepsilon=10^{-4}$. The initial samples are randomly chosen with values in $[-1,1]^M$. Results are shown in Figure \ref{fig:single-lloyd} and all simulations finished before the maximum number of iterations was attained. We have the following observations:
\begin{itemize}
\item The proposed method successfully approximates the Blaschke-Santal\'o diagrams even when using a rather small number of samples. At the end of the iterative process the images of the samples are uniformly distributed in the images $Y_d=F(\Sym_d([-1,1]))$.
\item Images of samples that are in the interior of the diagram $Y_d$ are close to the center of gravity of the corresponding Voronoi cell.
\item Images $y_i$ of samples lying on the boundary of $Y_d$ correspond to Voronoi cells $V_i$ where the distance between $y_i$ and the corresponding centroid $c_i$ is large.
\end{itemize}

\begin{figure}
	\centering
	\includegraphics[height=0.45\textwidth]{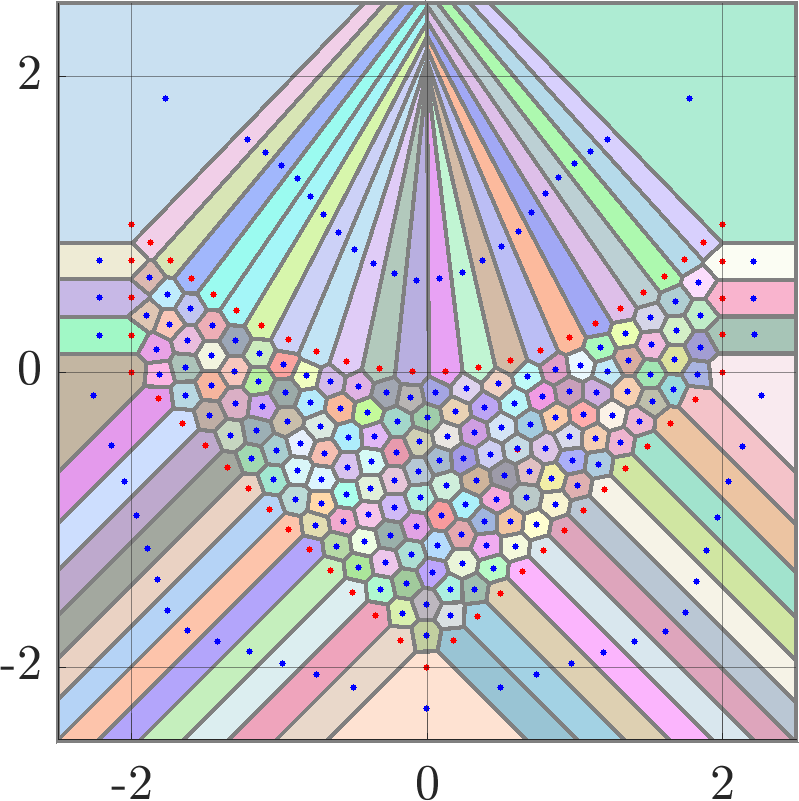}\quad
	\includegraphics[height=0.45\textwidth]{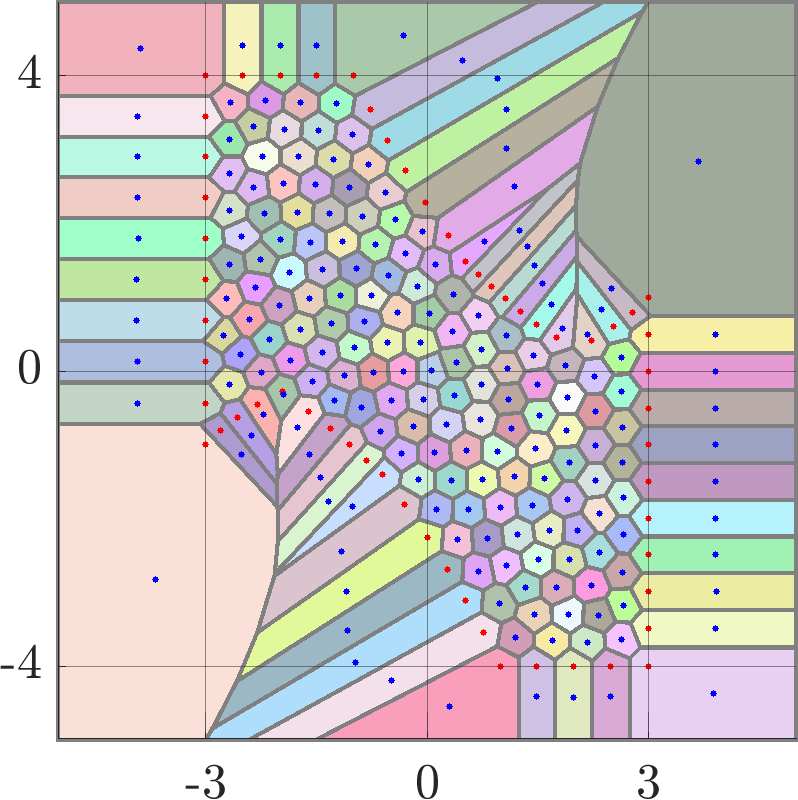}\quad
	\caption{Approximations of the $(\tr,\det)$ diagram in $\Sym_d([-1,1])$ for cases $d\in \{2,3\}$ using Algorithm \ref{algo:lloyd}.}
	\label{fig:single-lloyd}
\end{figure}

Next we apply Algorithm \ref{algo:var-cvt} for $M=200$ samples. The constrained numerical optimization problem is executed using the software Artelys Knitro with the \texttt{active-set} algorithm. The initial samples are randomly chosen with values in $[-1,1]^M$. The number of iterations is limited to $1000$ and the optimality criterion tolerance is set to $10^{-8}$. All simulations reached the maximum number of iterations under these constraints. Results are shown in Figure \ref{fig:single-cvt}. Similar observations can be underlined, recalling again that the diagrams are approximated remarkably well with $200$ well distributed samples. 

The function \eqref{eq:cvt-energy} is evaluated for the optimal configuration of both algorithms and shown in Table \ref{tab:comparison-energy}. It is clearly observed that Algorithm \ref{algo:var-cvt} provides a lower energy since this algorithm is focused on decreasing the objective value. Moreover, the use of quasi-Newton descent directions (compared to anti-gradient descent direction for Lloyd) is known to accelerate the convergence. On the other hand Lloyd's algorithm can make large changes in the optimization variables (each sample is replaced by another one closest to the corresponding centroid as possible) which is advantageous when applied to initial random configuration or to regions where \eqref{eq:cvt-energy} varies too slowly.

Concerning the computational cost, Algorithm \ref{algo:lloyd} is more costly (in our direct implementation), possibly due to the large number of small size optimization problems \eqref{eq:inverse-sample} that need to be solved at each iteration. Algorithm \ref{algo:var-cvt} simply computes one Restricted Voronoi Diagram per iteration (global or line-search) and evaluates the associated cost function and its gradient.

\begin{figure}
	\includegraphics[height=0.45\textwidth]{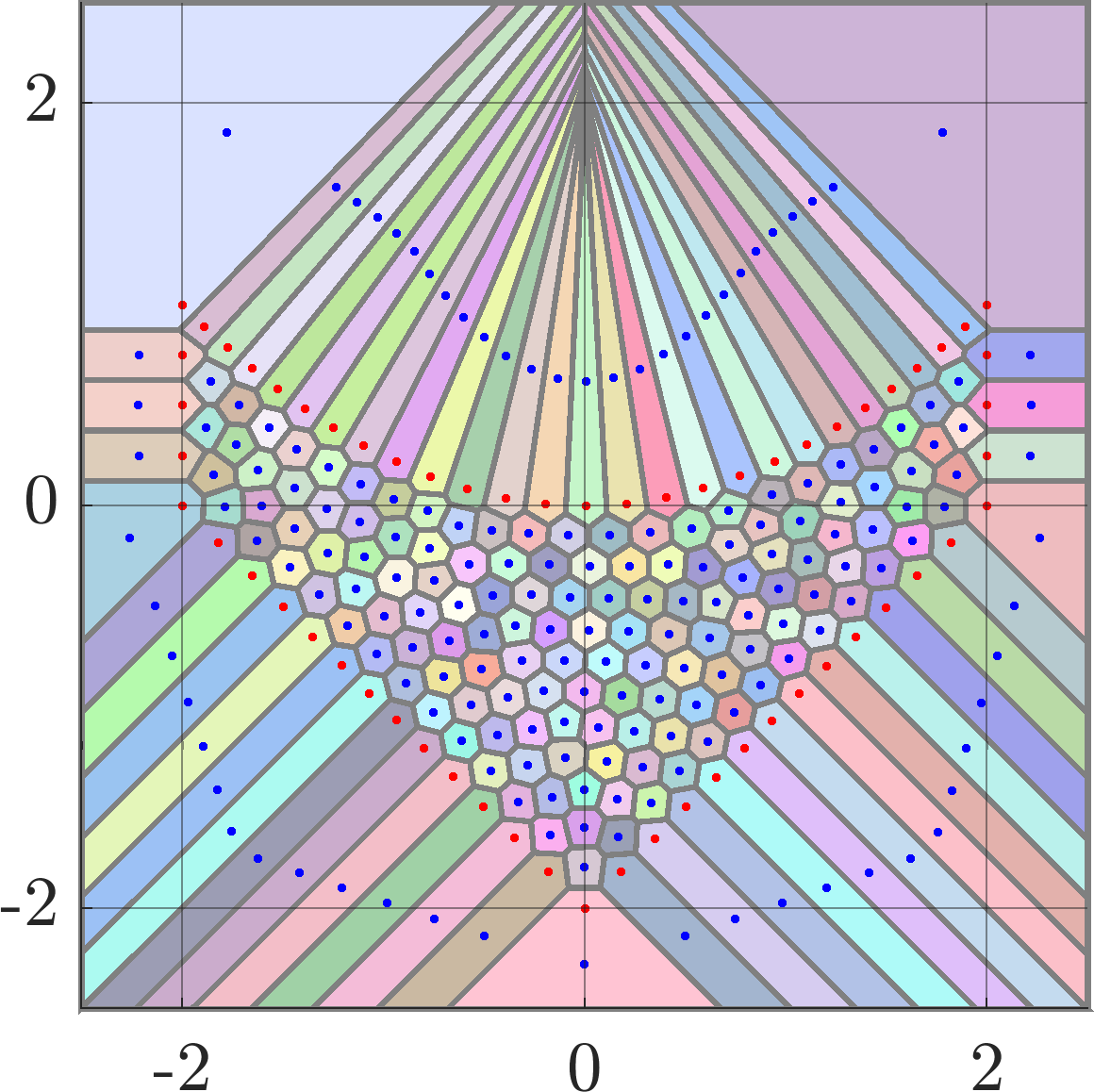}\quad
	\includegraphics[height=0.45\textwidth]{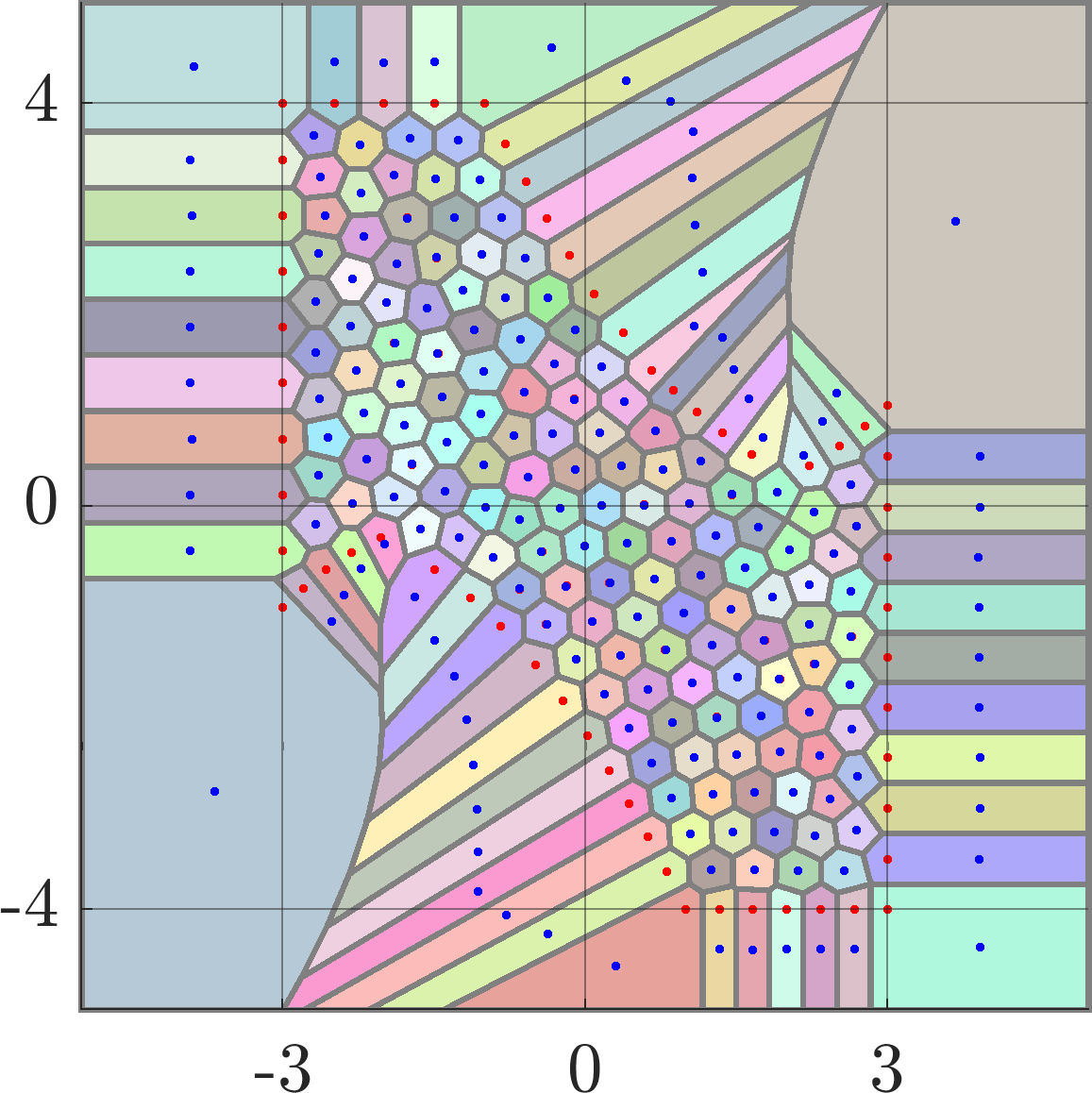}%\quad
	\caption{Approximations of the $(\tr,\det)$ diagram in $\Sym_d([-1,1])$ for cases $d\in \{2,3\}$ using Algorithm \ref{algo:var-cvt}.}
	\label{fig:single-cvt}
\end{figure}

\begin{table}
	\centering
	\begin{tabular}{|c|c|c|}
		\hline
     & Algorithm \ref{algo:lloyd} & Algorithm \ref{algo:var-cvt} \\ \hline
    $d=2$ & $19.249278$ & $19.246159$\\ \hline
    $d=3$ & $218.595970$ & $218.479052$ \\ \hline
    $d=4$ & $2359.898079$ & $2356.123844$ \\ \hline
	\end{tabular}
	\caption{Comparison of the energies \eqref{eq:cvt-energy} after $1000$ iterations of Algorithms \ref{algo:lloyd}, \ref{algo:var-cvt} for the final configurations shown in Figures \ref{fig:single-lloyd}, \ref{fig:single-cvt}.}
	\label{tab:comparison-energy}
\end{table}

Applying the proposed algorithms directly for a large number of samples is rather inefficient, convergence being slow. In the following we propose a multi-grid strategy that accelerates convergence.

%%%%%%%%%%%%%%%%%%%%%%%%%%%%%%%%%%%%%%%%%%%%%%%%%%
\section{Practical aspects: re-centering, multi-grid}

In many numerical applications, multi grid strategies are employed to accelerate simulations. An initial simulation is performed on a coarse grid. Then the discretization is refined and the current numerical optimizer is interpolated on the new grid and used as an initialization for a new optimization procedure. The refinement procedure is repeated until de desired precision is attained. 

We intend to use a similar strategy described below. Given a point $y = F(x) \in F(X)$ we search for other points $(z_i)_{i=1}^k$ in $X$ such that $F(z_i)$ is \emph{close} to $F(x)$ for $i=1,...,k$. A natural idea is to find an ellipsoid around $x$ which is mapped onto a $d$-dimensional sphere around $F(x)$. However, if $x$ is a boundary point for $F(X)$ it is not possible to find such an ellipsoid. Nevertheless, if the dimension $N$ of the space of parameters is larger than the dimension $d$ of the image, the Jacobian $DF$ will be most likely non-singular, thus allowing us to move the point in the interior of the constraint set $X$ while preserving the same image. This idea is detailed below.

\subsection{\bf Re-centering procedure.} Suppose $X = \prod_{i=1}^N [a_i,b_i]$, and $x_0 \in X$, $y_0=F(x_0)$ are such that $DF(x_0)$ is of full rank. The implicit function theorem implies that $\{F(x)=y_0 : x \in X\}$ is, locally around $x_0$, a $C^1$ hypersurface of dimension $N-d$. Supposing $N>d$ (always the case in our applications), if $\ker F(x_0)$ is not contained in the tangent space to $\{F(x)=y_0\}$ at $x_0$, then it is always possible to find $x$ in the interior of $X$ such that $F(x)=F(x_0)$. In particular, if $x_0$ is on the boundary of $X$, we can find another element in the interior of $X$, having the same image. 

In practice, we solve a problem of the form
\begin{equation}\label{eq:re-center}
\min_{ F(x)=F(x_0)} \left(\sum_{i=1}^N (x^i-0.5(a_i+b_i))^p\right)^{1/p}.
\end{equation}
Generic available software like \texttt{fmincon} in Matlab allows the implementation of the nonlinear constraint $F(x)=F(x_0)$. The power $p$ is chosen large enough ($p=10$ in practice) such that the maximal difference between the coordinates of $x$ and coordinates of the center of $X$ becomes as small as possible. Problems \eqref{eq:re-center} are computationally cheap for the algebraic application proposed previously. 

\subsection{\bf Multi-grid procedure.} Algorithms \ref{algo:lloyd} and \ref{algo:var-cvt} proposed in Section \ref{sec:framework} can converge slowly when the number of samples is large. This is especially problematic when random samples tend to concentrate mostly in some particular regions of the Blaschke-Santal\'o diagram. Thus we are interested in ways of enriching the set of samples given by Algorithms \ref{algo:lloyd} or \ref{algo:var-cvt} for a rather small initial number of samples. We propose two refinement procedures below.

{\bf (a) Spheres around current samples.} Suppose $N>d$. For a point $x_0$ such that $DF(x_0)$ is not singular, we re-center it using \eqref{eq:re-center}. We start by computing the singular value decomposition
\[ DF(x_0) = USV^T,\]
with $U \in \Bbb{R}^{d\times d}$, $S \in \Bbb{R}^{d\times N}$, $V \in \Bbb{R}^{N\times N}$. Matrices $U$ and $V$ are unitary and $S$ is diagonal, containing the singular values on the diagonal. Assuming $DF(x_0)$ is not singular, the diagonal values in $S$ are non-zero. Moreover, if $(u_i)_{1\leq i \leq d}, (v_j)_{1 \leq j \leq N}$ are columns if $U$ and $V$, respectively and $s_i$, $1\leq i \leq d$ are the singular values, we have the decomposition
\[ DF(x_0) = \sum_{i=1}^d s_i u_i v_i^T.\]
Denote by $\overline V$ the matrix containing the first $d$ columns of $V$ as columns and $\overline S$ the $d\times d$ diagonal matrix containing $1/s_i$ on its diagonal.
Consider the vectors $w_i$, $1\leq i \leq d$ which are columns of $ \overline V\cdot\overline S\cdot U$. Then these vectors verify $DF(x_0)w_i = e_i$, $i \leq i \leq d$ where $(e_i)_{i=1,...,d}$ is the canonical basis. Consider the matrix $W$ containing $w_i$, $i=1,...,d$ as columns. For $k\geq 3$ consider $k$ uniformly distributed points $(\theta_j)_{j=1}^k$ in $\Bbb{S}^{d-1}$ and consider points $z_j= x_0 + rW\theta_j$, where $r>0$ is a radius small enough. In practice $r$ is equal to one third of the minimal distance among images $F(x_i)$. Among points $z_j$ we select only those that belong to $X$. If $x_0$ is a boundary point for $F(X)$ it is possible that only a few of the points $z_j$, $j=1,...,d$ are admissible.

Whenever a refinement is necessary, the procedure described above is repeated for every sample point $(x_i)_{i=1}^M$ for which the singular values of $DF(x_i)$ are above a certain threshold ($10^{-3}$ in our implementation). 

In Figure \ref{fig:example-ref} an example of application of the methods described above is shown. We take a result given by Algorithm \ref{algo:var-cvt}. At the end of the optimization process some samples may have images at the boundary of $X$. Applying directly the multi-grid procedure, trying to add points on a circle around the current samples gives the result shown in the left picture in Figure \ref{fig:example-ref}. For some points in the interior of $F(X)$ the algorithm fails to add the required number of points since the associated samples are on the boundary of $X$. In the right picture, re-centering is performed before applying the refinement procedure. For all interior points of $F(X)$ the algorithm manages to add the prescribed number of additional samples with images close to the previous ones. It can be observed that for the upper boundary, corresponding to a singular Jacobian in this case, no points are added, since the procedure described above cannot be applied.

\begin{figure}
	\centering
	\fbox{\includegraphics[width=0.45\textwidth]{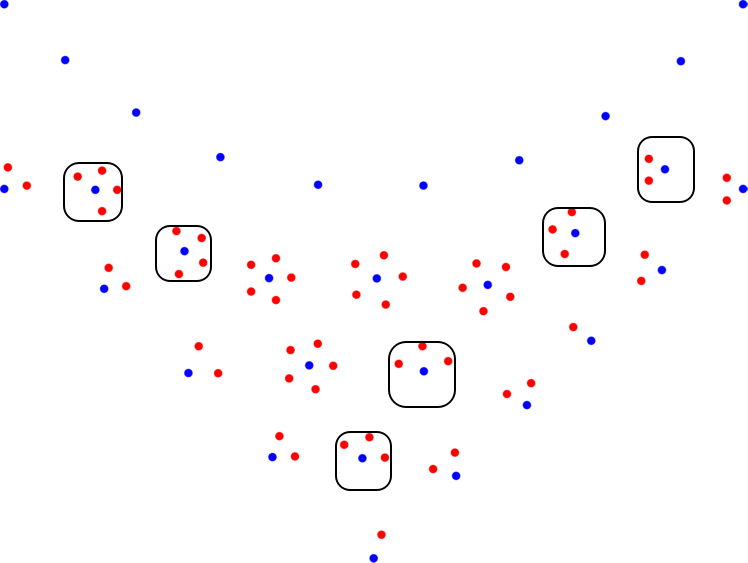}}\qquad
	\fbox{\includegraphics[width=0.45\textwidth]{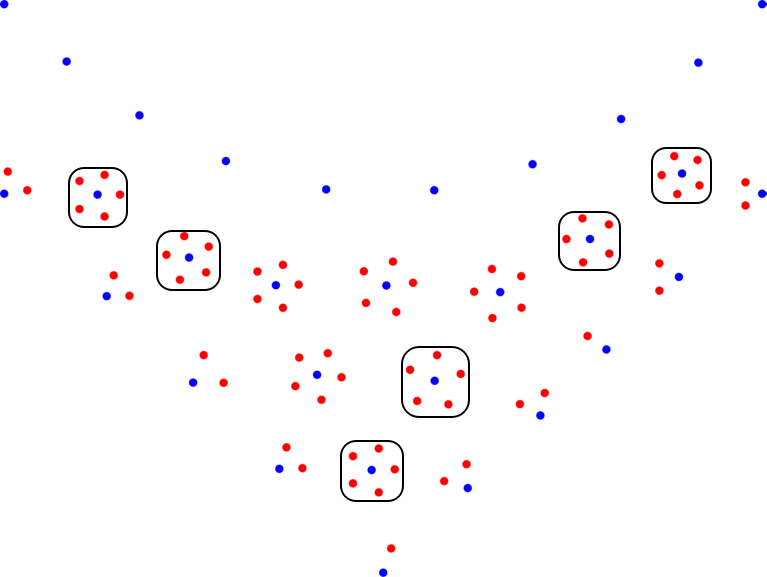}}
	\caption{Refinement procedure: application before re-centering fails to add the required number of points around all interior points in the image. (left) After re-centering, the multi-grid procedure succeeds for all interior points. (right)}
	\label{fig:example-ref}
\end{figure}

{\bf (b) Delaunay Triangulations.} These triangulations are closely related to Voronoi diagrams and their computation is standard in computational geometry. The Delaunay triangulation is the dual graph of the Voronoi diagram, the circumcenters of the Delaunay triangles being the vertices of the Voronoi diagram. 

\begin{figure}
	\includegraphics[width=0.5\textwidth]{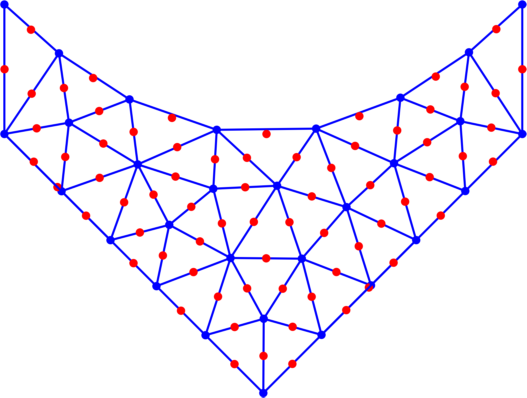}
	\caption{Refinement strategy using projections of midpoints of certain edges of the triangles in the Delaunay triangulation.}
	\label{fig:delaunay}
\end{figure}

Given a set of samples $(x_i)_{i=1}^M$ and the corresponding images $y_i=F(x_i)$, $i=1,...,M$, start by computing the Delaunay triangulation $\mathcal T$ of $(y_i)_{i=1}^M$. Next, select all edges of triangles in $\mathcal T$. For each such edge, take the midpoint $y_m$ and solve the problem 
\begin{equation}\label{eq:projection}
\min_{x \in X} \|F(x)-y_m\|.
\end{equation}
If the minimization \eqref{eq:projection} succeeds, the solution $x_m$ is a sample for which $F(x)=y_m$. If $y_m$ does not belong to $F(X)$ then the minimization problem still produces a sample, as close as possible to $y_m$. There are two issues that motivate us to solve \eqref{eq:projection} only for particular edges in $\mathcal T$.
\begin{itemize}
	\item When $F(X)$ is non-convex, the Delaunay triangulation $\mathcal T$ will contain triangles which are not contained in $F(X)$. Usually, such triangles have an obtuse angle, some sides being significantly larger than the others.
	\item Some regions may contain a denser concentration of samples than others. Therefore, for edges with length too small compared to the average, we choose not to add the corresponding midpoints to the diagram.
\end{itemize}
In practice we compute the average length $\ell$ of sides of triangles in $\mathcal T$ and we solve \eqref{eq:projection} for midpoints of edges with length in $[0.5\ell,1.5\ell]$. Figure \ref{fig:delaunay} shows the outcome of this refinement procedure for the same test case as the one showin in Figure \ref{fig:example-ref}. 

\subsection{\bf Global algorithm and numerical examples.} Taking the previous considerations into account leads us to the following practical approach. In particular, we combine the advantages of Algorithms \ref{algo:lloyd}, \ref{algo:var-cvt} and the refinement strategy. Algorithm \ref{algo:var-cvt} is ran until convergence criteria are met or the maximal number of iterations is reached.

\begin{algo}[Blaschke-Santal\'o multi-grid approximation]
	Inputs: $M$ - number of initial samples, maximal number of iterations $q_1 \in [20,100]$ for Algorithm \ref{algo:lloyd} and $q_2\in [1000,2000]$ for Algorithm \ref{algo:var-cvt}, number of refinements $n_{\reff}$, number of points to add around each sample at refinement $n_{\addd}$.
	
	Initialization: Choose $M$ random samples in $X$. Run Algorithm \ref{algo:lloyd} for $q_1$ iterations followed by Algorithm \ref{algo:var-cvt}.
	
	For $k=1$ to $n_{\reff}$ do:
	\begin{itemize}
		
		\item Multigrid: do one of the following. 
		\begin{itemize}
			\item Apply the re-centering procedure, solving \eqref{eq:re-center} for each one of the resulting samples. Apply the first refinement procedure adding at most $n_{\addd}$ points around each sample. 
			\item Alternatively, use the midpoints of the triangles in the Delaunay triangulation to find new sample points.
		\end{itemize}
		\item Run Algorithm \ref{algo:lloyd} for $q_1$ iterations.
		\item Run Algorithm \ref{algo:var-cvt}.
		%\item Run Algorithm \ref{algo:lloyd} for $q_1$ iterations.
	\end{itemize}
\label{algo:global-ref}
\end{algo}
\vspace{0.2cm}

Running a few iterations of Algorithm \ref{algo:lloyd} before running Algorithm \ref{algo:var-cvt} is motivated by the fact that Lloyd's algorithm can make large jumps when applied to a non-optimal initial configuration.

In the following we show how Algorithm \ref{algo:global-ref} approximates Blascke-Santal\'o diagrams in practice. First we apply it for the $(\tr,\det)$ diagram in the case $d=2$. We start with a set of $30$ samples and we perform three refinements. The simulations have $30$, $104$, $464$ and $2423$ samples, respectively. Initialization and the results of the successive stages of the algorithm are shown in Figure \ref{fig:global-ref-mat-2}.

\begin{figure}
	\centering
	\includegraphics[width=0.3\textwidth]{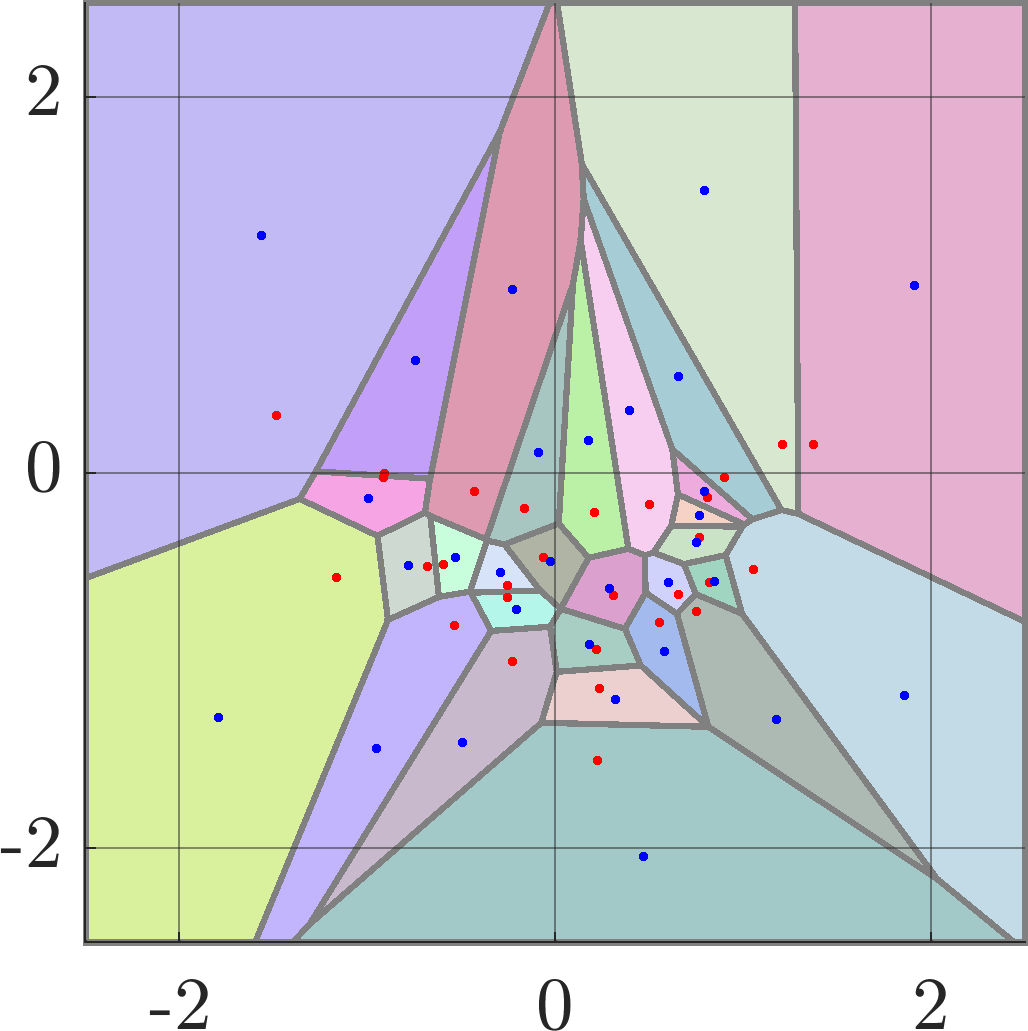}\quad
	\includegraphics[width=0.3\textwidth]{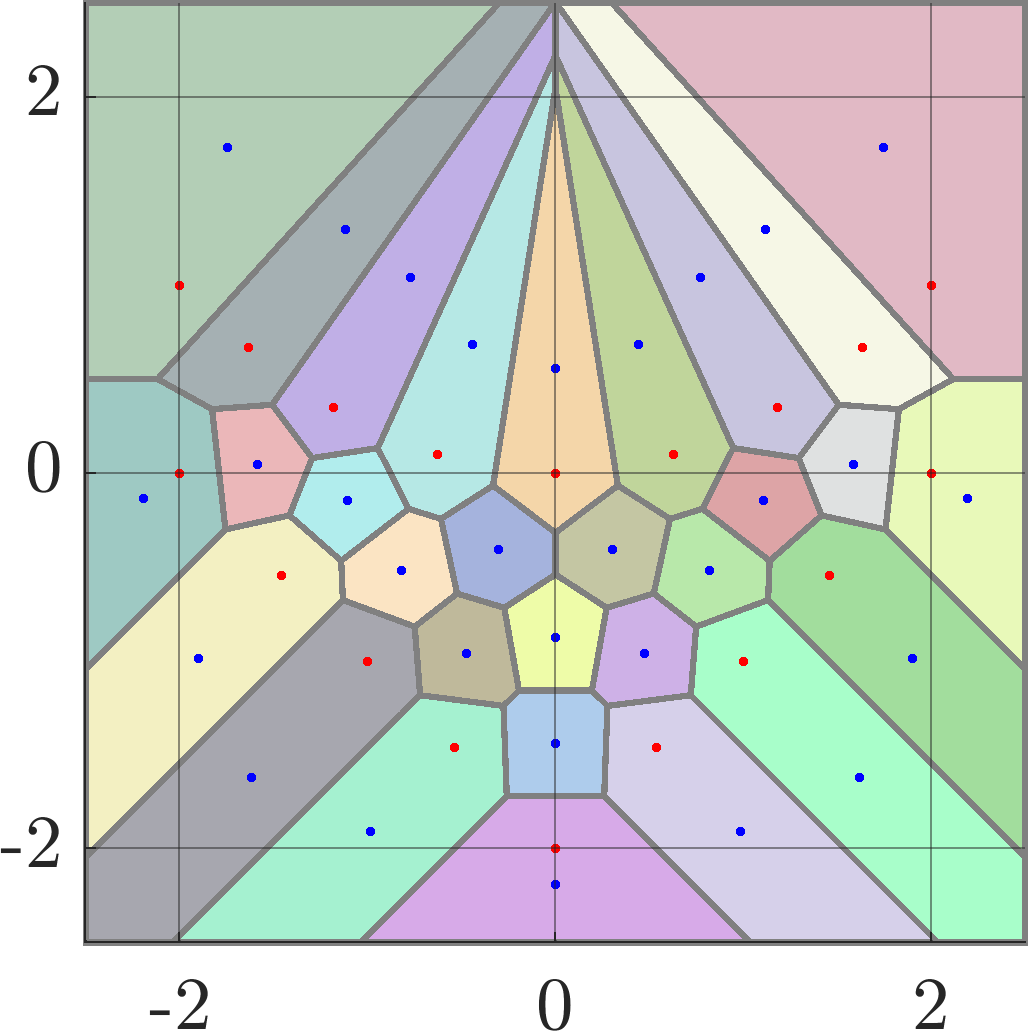}\quad
	\includegraphics[width=0.3\textwidth]{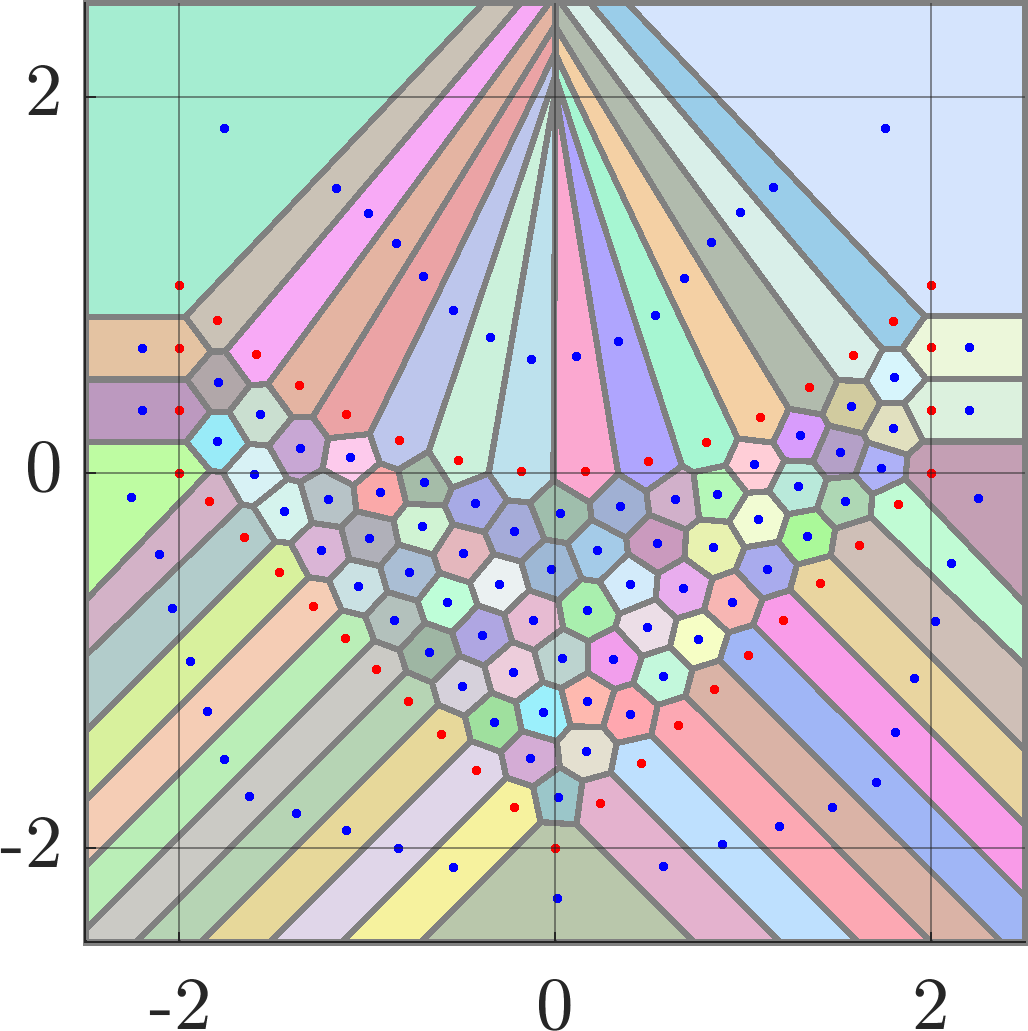}
	
	\vspace{3pt}
	
	\includegraphics[width=0.48\textwidth]{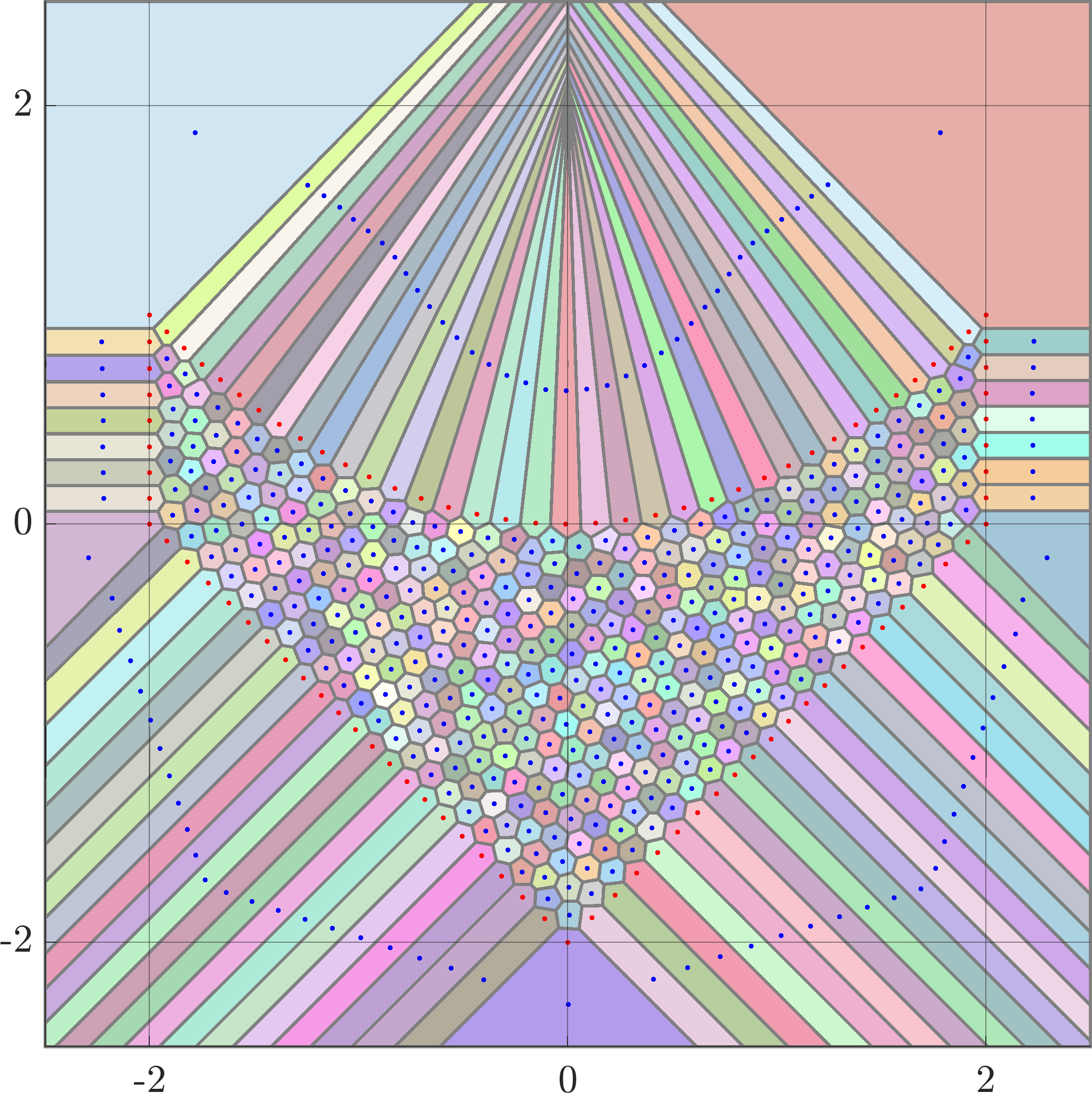}\quad
	\includegraphics[width=0.48\textwidth]{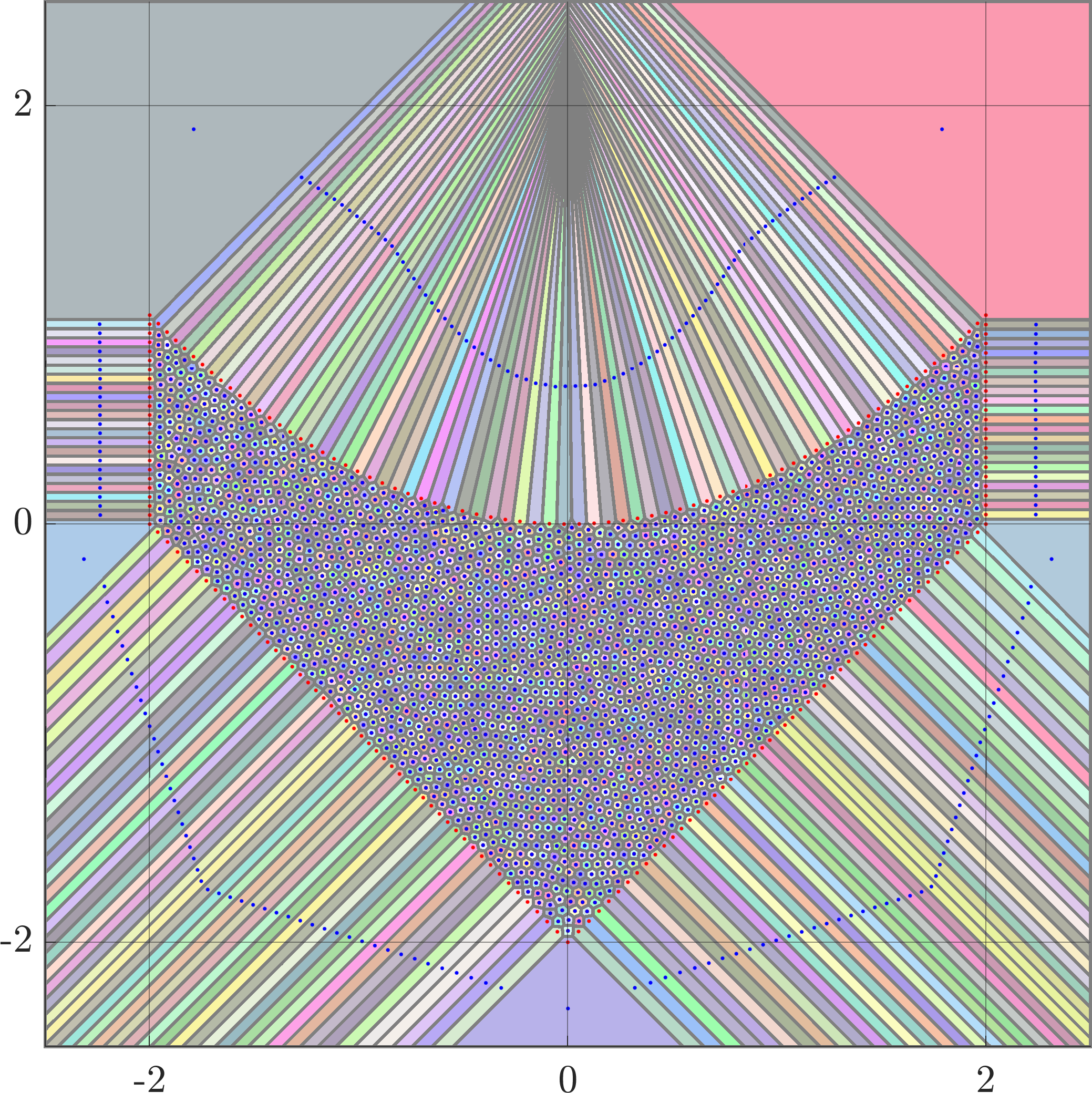}
	\caption{The case $d=2$ initialization and optimized configurations after each refinement. The simulations have $30$, $104$, $464$ and $2423$ samples, respectively.}
	\label{fig:global-ref-mat-2}
\end{figure}

Similar simulations are made for $d\in \{3,4\}$. The resulting final optimized configurations having $2684$ and $2043$ cells, respectively, are shown in Figure \ref{fig:results-3-4}. %Comparing these results with the random samples shown in Figure \ref{fig:random} shows the advantage of the method proposed here. Using significantly fewer samples, the Blaschke-Santal\'o diagram can be well approximated. 

\begin{figure}
\includegraphics[height=0.7\textwidth]{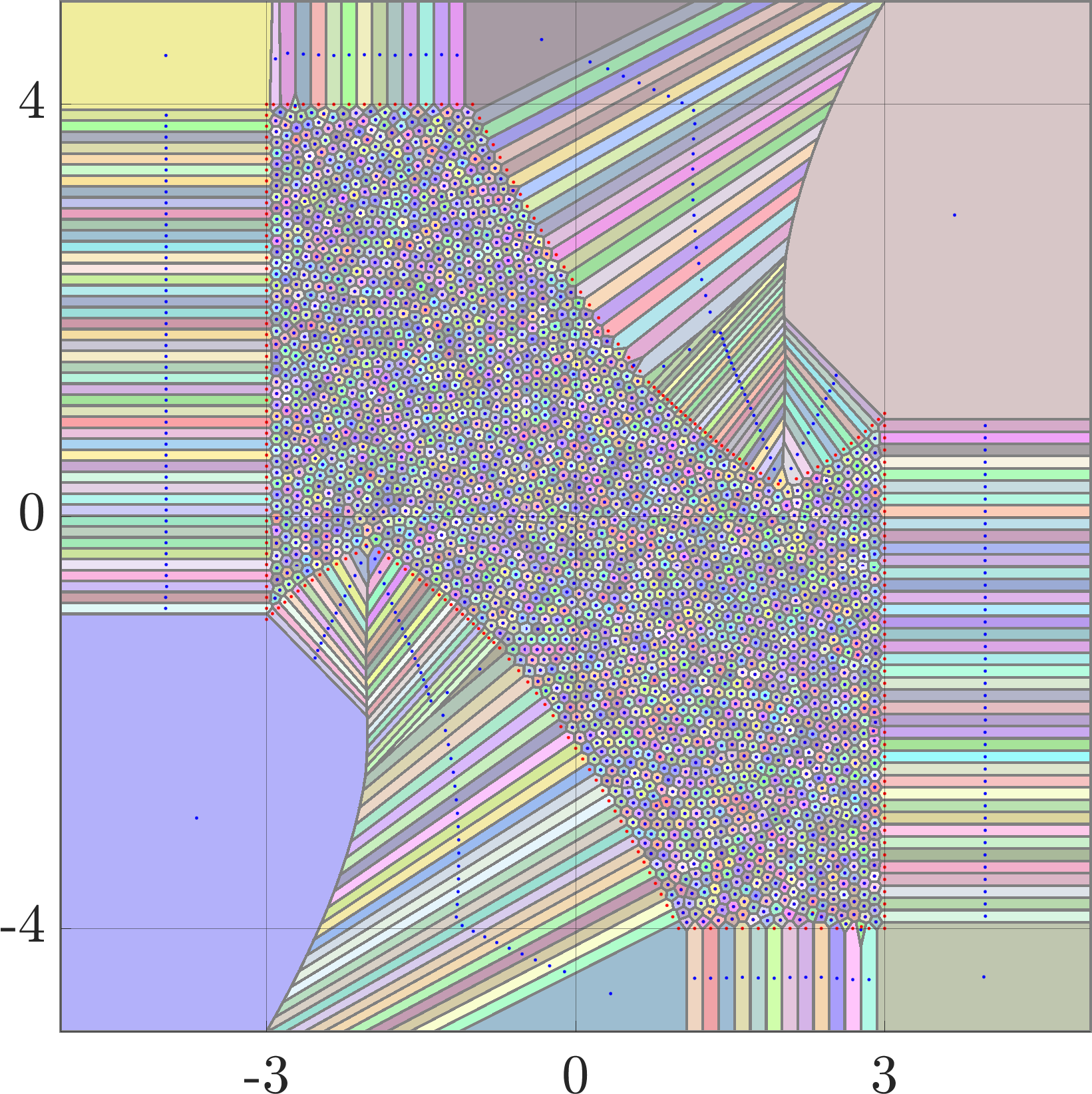}\quad
\includegraphics[height=0.7\textwidth]{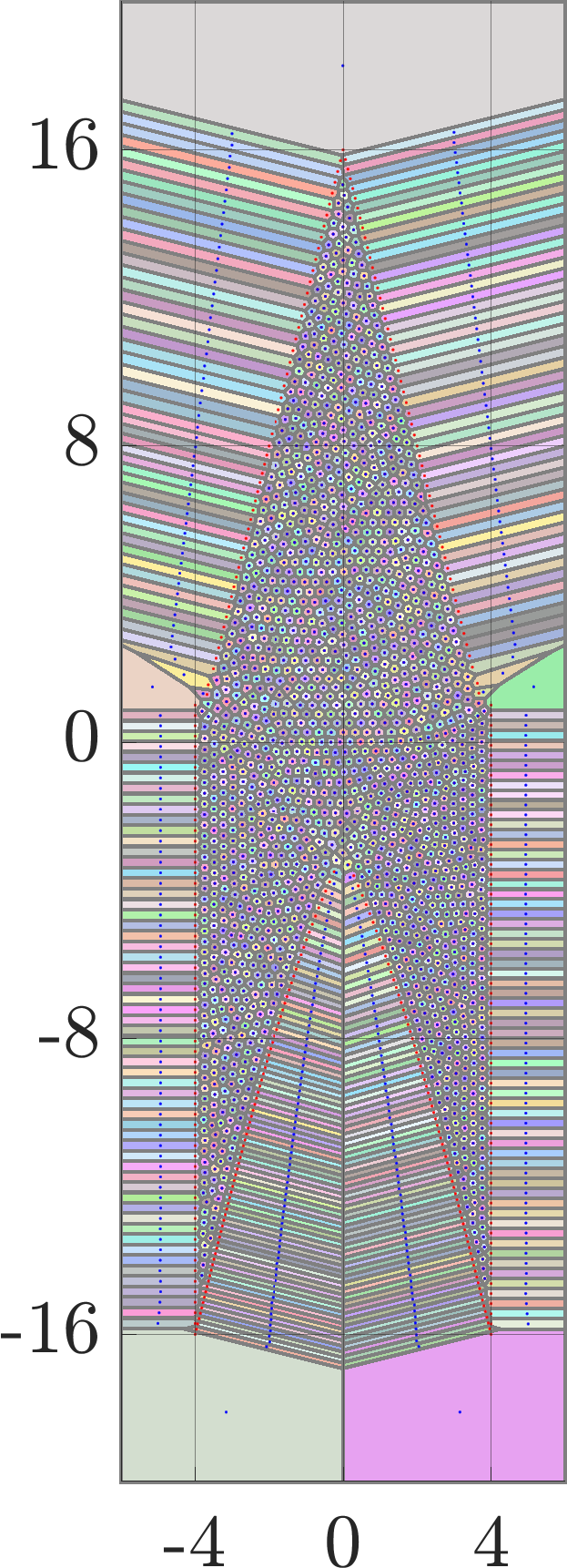}
\caption{Results given by Algorithm \ref{algo:global-ref} for cases $d=3$ ($2684$ cells) and $d=4$ ($2043$ cells).}
\label{fig:results-3-4}
\end{figure}

In all results presented up to this point, the centroids of Voronoi cells are shown in blue and the images of the samples are shown in red. In some situations, for $d \in \{3,4\}$ we may observe inner Voronoi points which do not coincide with the Voronoi cell's centroid. This happens especially close to the curves parametrized by $[-1,1] \ni t \mapsto (dt,t^d)$ corresponding to diagonal matrices, for which the Jacobian of $(\tr,\det)$ is singular. Such points generate boundary behavior in the interior of the diagram. 

\subsection{\bf Extracting the Blaschke-Santal\'o diagram from the Voronoi diagram.} The boundary points of the Blaschke-Santal\'o diagram can be recovered selecting only samples for which the image is far from the centroid of the associated Voronoi cell. However, extracting a polygon from these points is not straightforward. We use the following ideas to plot the Blaschke-Santal\'o diagram starting from the numerical results:
\begin{itemize}
\item If the resulting diagram is convex, taking the convex hull of the images of the samples suffices. 
\item Since the optimized samples form a Centroidal Voronoi tessellation, except the boundary points, we exploit the associated Delaunay triangulation which covers the entire convex hull of the diagram. However, triangles which are outside our diagram are nearly flat (having a small or large angle). We eliminate from the Delaunay triangulation such triangles (with thresholds that are set case by case).
\end{itemize}
The resulting diagrams are shown in Figure \ref{fig:Mat-extracted}.

\begin{figure}
\centering
\includegraphics[height=0.31\textheight]{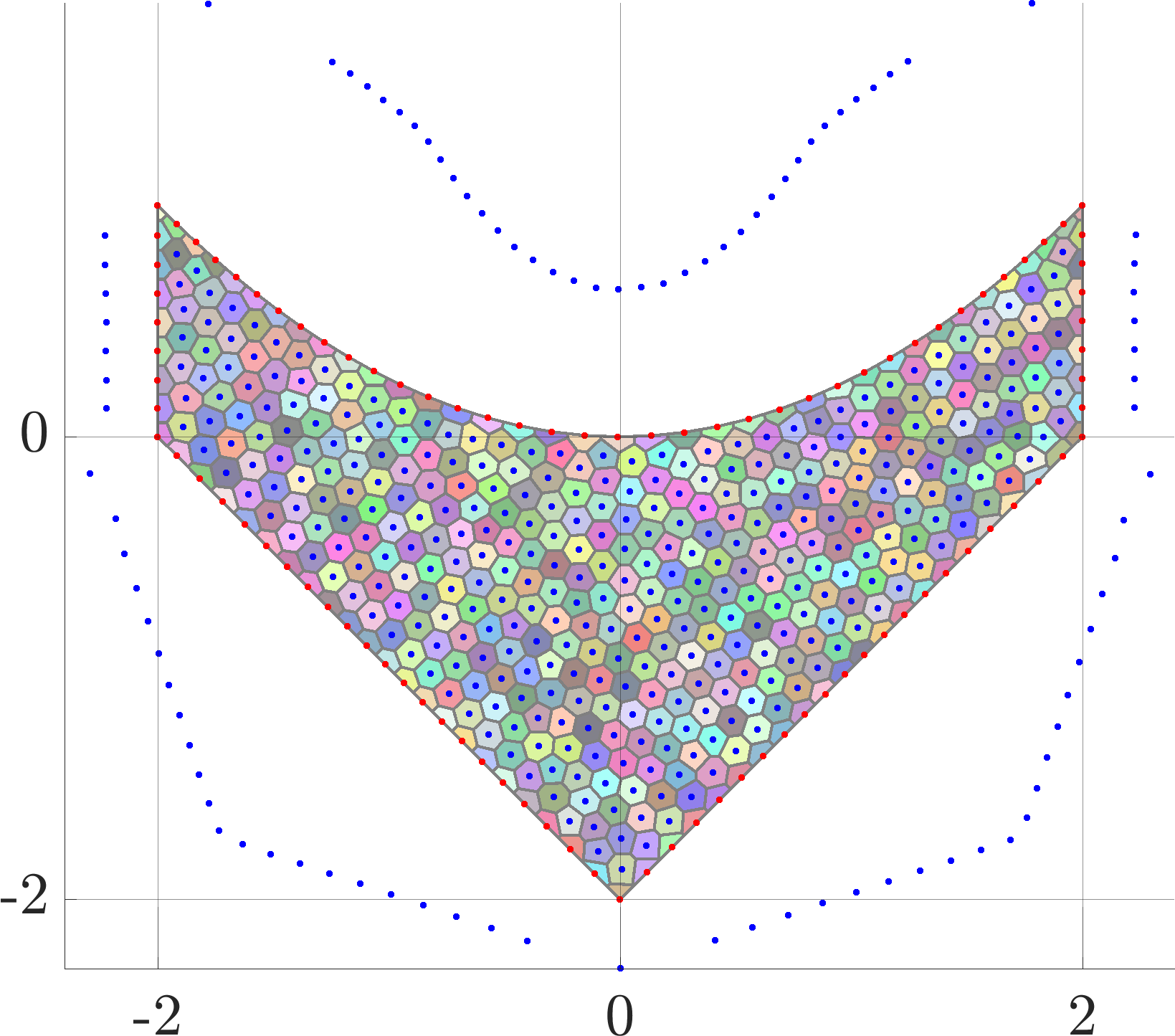}\quad
\includegraphics[height=0.31\textheight]{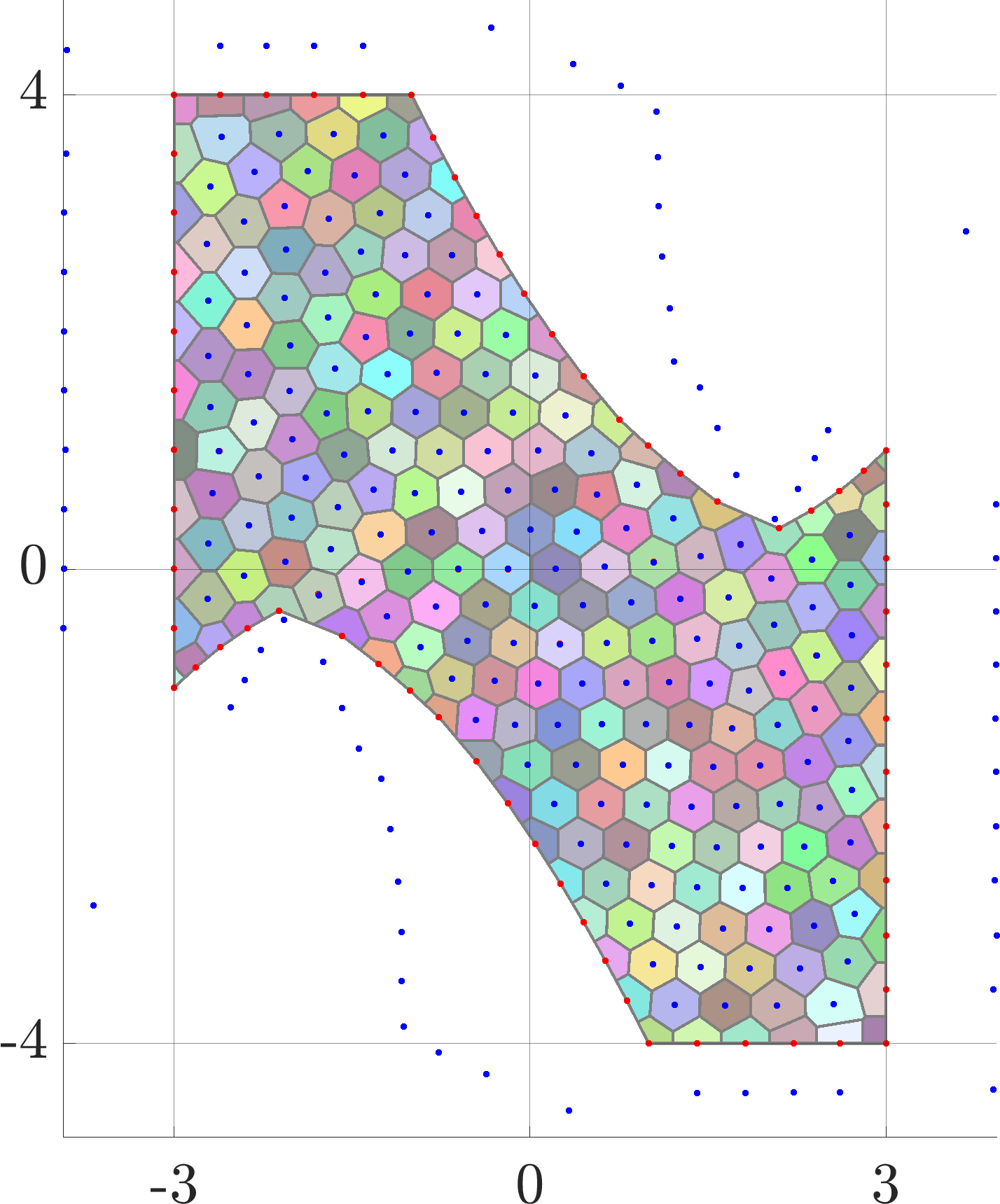}
\caption{Extracting the Blaschke-Santal\'o diagram from the optimal Voronoi tesellation by eliminating "flat" triangles from the associated Delaunay triangulation.}
\label{fig:Mat-extracted}
\end{figure}

\subsection{\bf Subset of a diagram.} In some cases, more detail is needed regarding certain parts of a Blaschke-Santal\'o diagram. Rather than increasing the point density everywhere in the diagram, it is possible to focus only on the region of interest. Suppose we are interested in the region $F(X) \cap B(y,r)$. Then we can implement all algorithms presented previously adding the additional constraint 
\begin{equation} |F(x_i)-y|\leq r.
\label{eq:restriction}
\end{equation}
The practical difficulty is that constraints \eqref{eq:restriction} are nonlinear. General software like \texttt{fmincon} and Knitro allow the use of non-linear constraints. The behavior of the optimization algorithm is improved if the gradient of the non-linear constraints is computed explicitly, which is possible in our case. Nevertheless, adding the non-linear constraints \eqref{eq:restriction} slows down the proposed algorithms. The speed loss is compensated by a lower number of samples, since we only focus on a subset of the desired diagram.

As an example, we show the of the $(\tr,\det)$ diagram for $d=4$ contained in the disk $B((3.9,1),0.2)$. We were interested in exploring extremal matrices in this region of the diagram, since the boundary parametrization seemed to change here. 

\begin{figure}
\centering
\fbox{\includegraphics[height=0.55\textwidth]{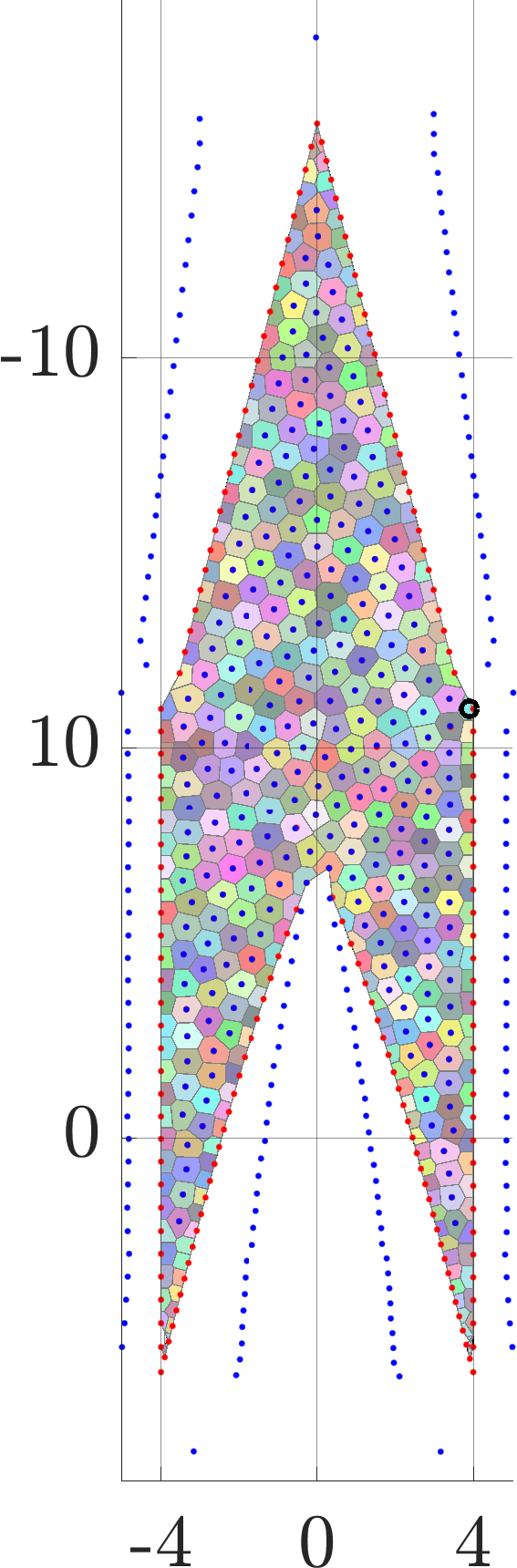}}\quad 
\fbox{\includegraphics[height=0.55\textwidth]{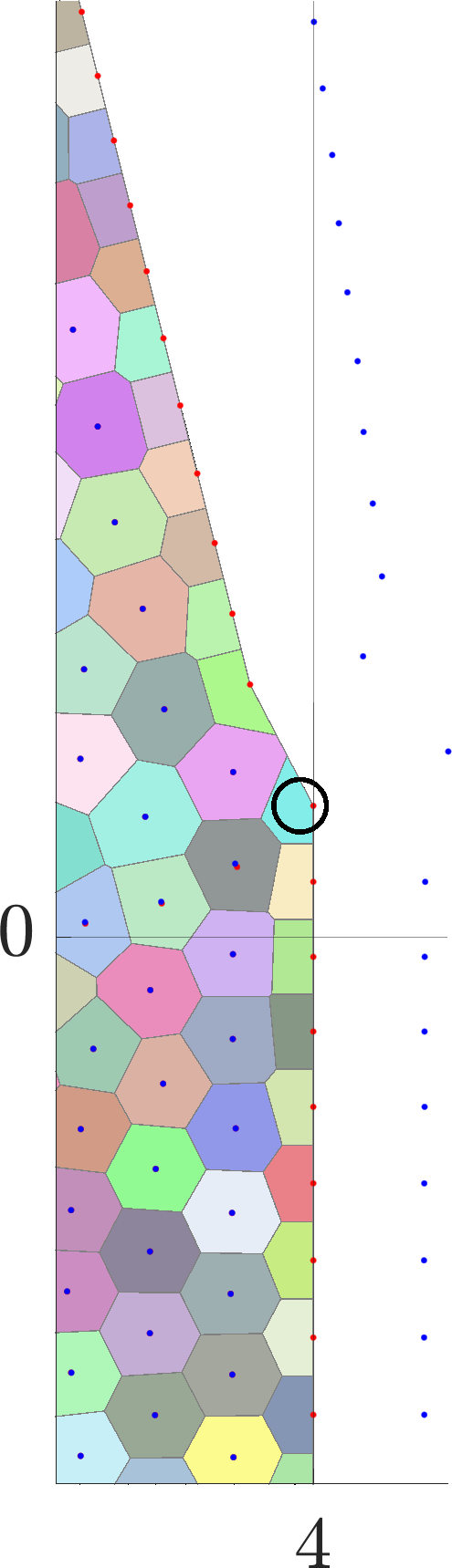}}\quad 
\fbox{\includegraphics[height=0.55\textwidth]{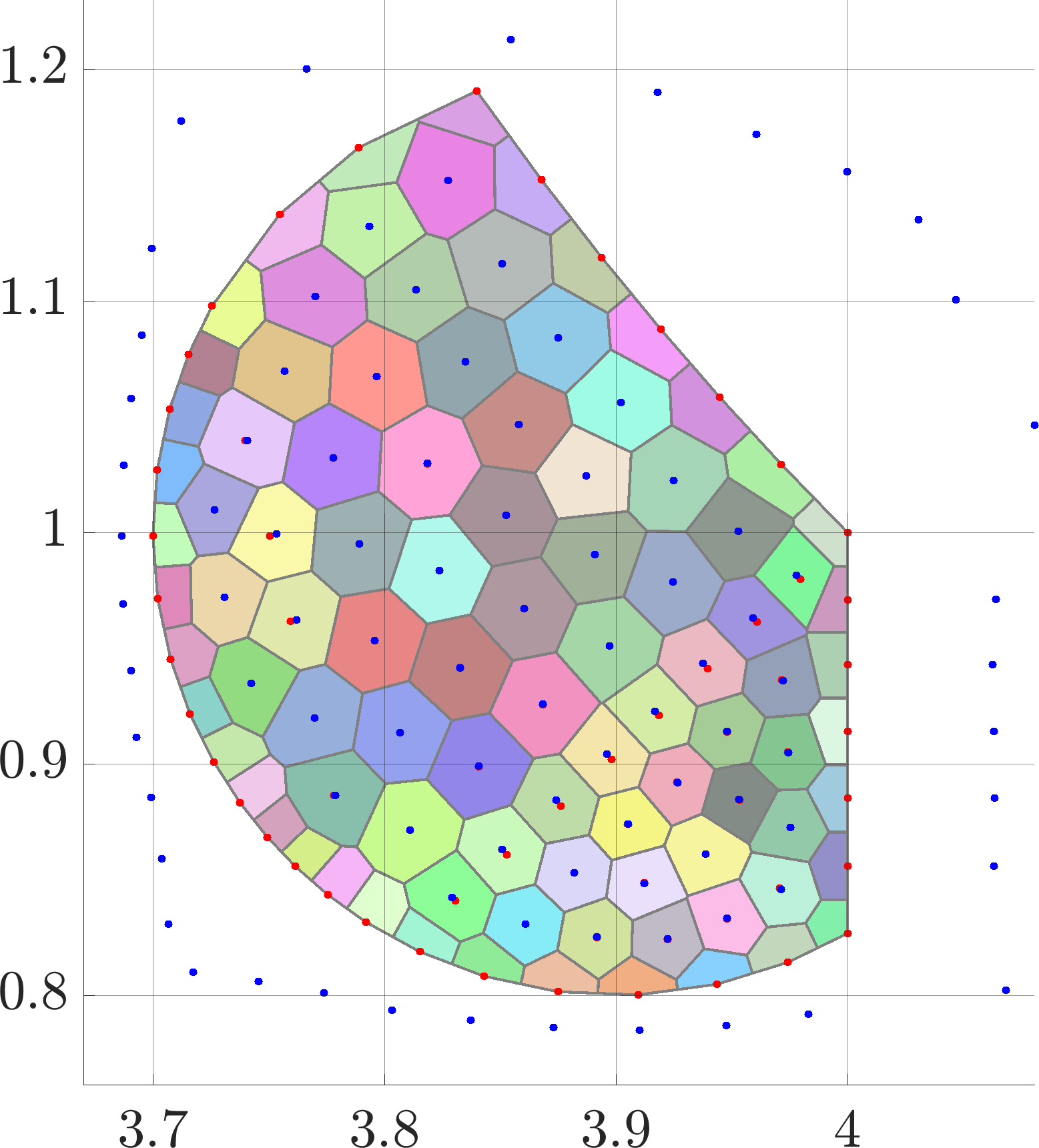}}
\caption{Blaschke-Santal\'o diagram for $(\tr,\det)$, $d=4$ (left), zoom around the point $(4,1)$ (center). Diagram restrained to the disk $B((3.9,1),0.2)$ (corresponding to the small circle in the center image) obtained using constraint \eqref{eq:restriction} in the algorithms (right).}
\label{fig:BS4_zoom}
\end{figure}

\subsection{\bf Details concerning $d \in \{3,4\}$.} We already saw that the case $d=2$ admits an explicit characterization of the boundary of the $(\tr,\det)$ diagram. For higher dimensions such a description is difficult to obtain. Nevertheless, identifying the extremal matrices for diagrams given in Figure \ref{fig:results-3-4} we are able to conjecture a precise parametrization of the corresponding boundaries. 

We give a brief description of the extremal matrices and corresponding parametrizations of the boundary for $d=3$:
\begin{itemize}
	\item Left boundary: $A = \begin{pmatrix}
		-1 & a & a \\
		a & -1 & -a \\
		a & -a & -1
	\end{pmatrix}$:   $[-1,1] \ni x\mapsto (-3,- 2x^3 + 3x^2 - 1)$
	\item Right boundary: $A=\begin{pmatrix}
	1 & a & a \\
	a & 1 & a \\
	a & a & 1
	\end{pmatrix}$: $[-1,1] \ni x \mapsto (3,2x^3 - 3x^2 + 1)$
	\item Top boundary part 1: $A=\begin{pmatrix}
	-1 & 1 & 1 \\
	1 & -1 & 1 \\
	1 & 1 & a
	\end{pmatrix}$:  $[-1,1]\ni x\mapsto (-2+x,4)$
	\item Top boundary part 2:  $A=\begin{pmatrix}
	1 & 1 & -1 \\
	1 & a & -1 \\
	-1 & -1 & a
	\end{pmatrix}$: $[-1,-0.125] \ni x\mapsto (2x+1,x^2 - 2x + 1)$
	
	\item Top boundary part 3: $\begin{pmatrix}
	a & -1 & -1 \\
	-1 & a & 1 \\
	-1 & 1 & a
	\end{pmatrix}$: $[0.25,2/3]\ni x\mapsto (3x,x^3 - 3*x + 2)$
	
	\item Top boundary part 4: $[2/3,1]\ni x \mapsto (3x,x^3)$
\end{itemize}
Similar observations can be made for $d=4$. Due to symmetry reasons, we only detail the right-half of the boundary:

\begin{itemize}
	\item Right boundary: $\begin{pmatrix}
	1 & -a & -a & a \\
	-a & 1 & -a & a \\
	-a & -a & 1 & a \\
	a & a & a & 1
	\end{pmatrix}$:   $[-1,1] \ni x\mapsto (4,- 3x^4 - 8x^3 - 6x^2 + 1)$
	\item Bottom right: 
	$\begin{pmatrix}
	a & -1 & 1 & 1 \\
	-1 & a & 1 & 1 \\
	1 & 1 & a & -1 \\
	1 & 1 & -1 & a
	\end{pmatrix}$: $[-1,1] \ni x \mapsto (4x,x^4 - 6x^2 - 8x - 3)$
	\item Top right-part 1: $\begin{pmatrix}
	1 & -1 & 1 & -1 \\
	-1 & a & -1 & -1 \\
	1 & -1 & a & 1 \\
	-1 & -1 & 1 & 1
	\end{pmatrix}$: $[-1,\frac{3-\sqrt{2}}{2}] \ni x \mapsto (2x+2,8-8x)$
	\item Top right-part 2: $\begin{pmatrix}
	1 & 1 & -1 & 1-a \\
	1 & a & 0 & 1 \\
	-1 & 0 & 1 & 1 \\
	1-a & 1 & 1 & 1
	\end{pmatrix}$: $[2-\sqrt{2},1]\ni x\mapsto (x+3,-(x - 2)(x^2 - 2x + 2))$
\end{itemize}
The resulting parametrized boundaries are shown in Figure \ref{fig:d=3,4-param}.
\begin{figure}
	\centering
	\includegraphics[height=0.5\textwidth]{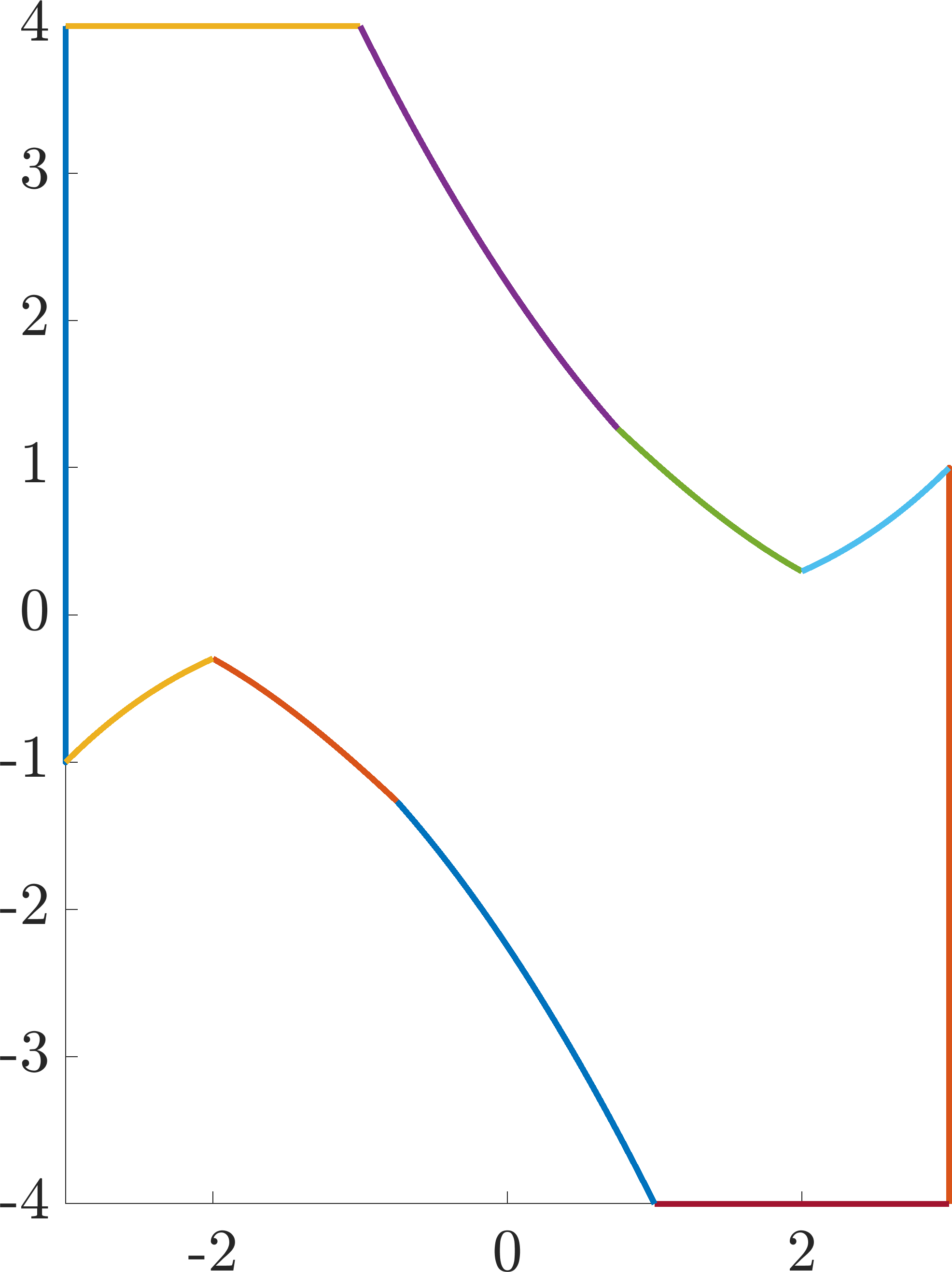}\quad
	\includegraphics[height=0.5\textwidth]{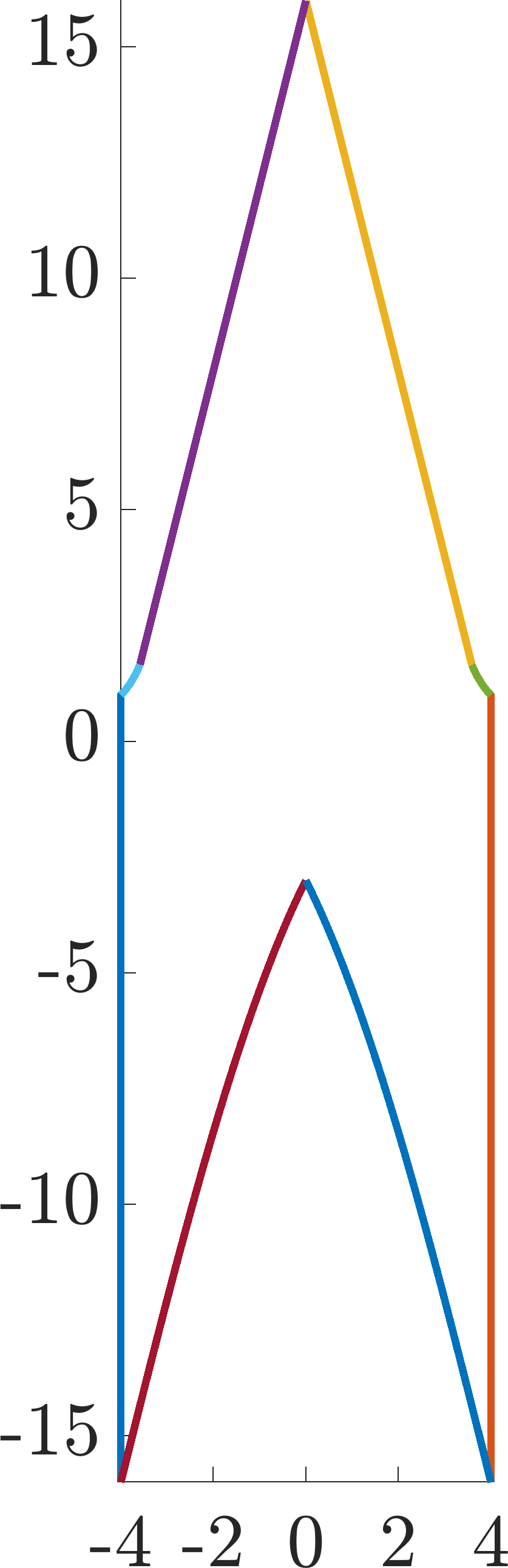}
	\caption{Direct parametrizations of the boundaries of the $(\tr,\det)(\Sym_d([-1,1])$ diagrams for $d=3,4$.}
	\label{fig:d=3,4-param}
\end{figure}

\subsection{Higher dimensions.} The algorithm proposed generalizes in a straightforward way to three dimensional diagrams. All algorithmic aspects remain the same. Three dimensional Restricted Voronoi diagrams are computed again using the library Geogram \cite{cvt-levy}. The example for the diagram $(\tr, \lambda_1\lambda_2+\lambda_2\lambda_3+\lambda_3\lambda_1,\det)$ for matrices in $\Sym_3([-1,1])$ is shown in Figure \ref{fig:3D-computations}. As usual, $\lambda_1,\lambda_2,\lambda_3$ denote the eigenvalues of the $3\times 3$ matrix. 
\begin{figure}
\centering
\includegraphics[width=0.45\textwidth]{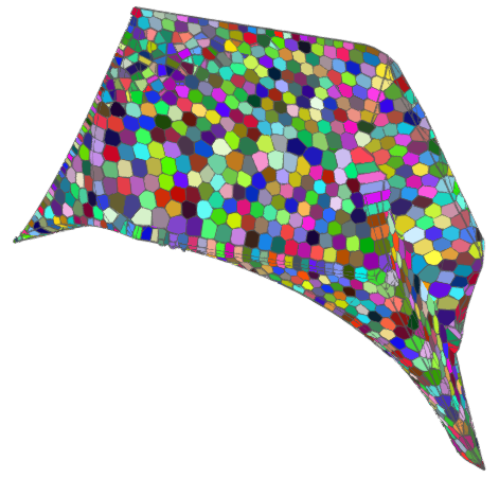}\quad
\includegraphics[width=0.45\textwidth]{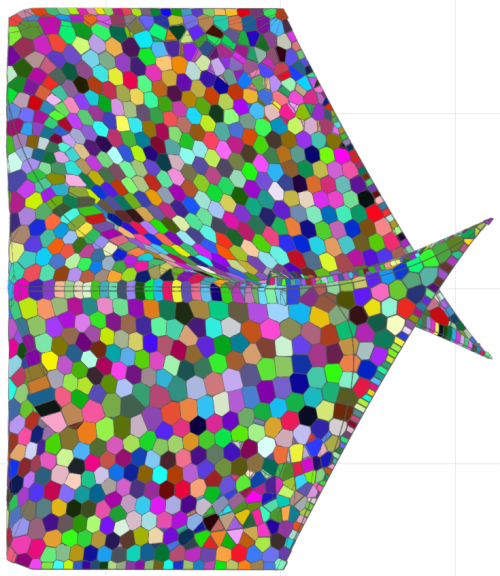}
\includegraphics[width=0.45\textwidth]{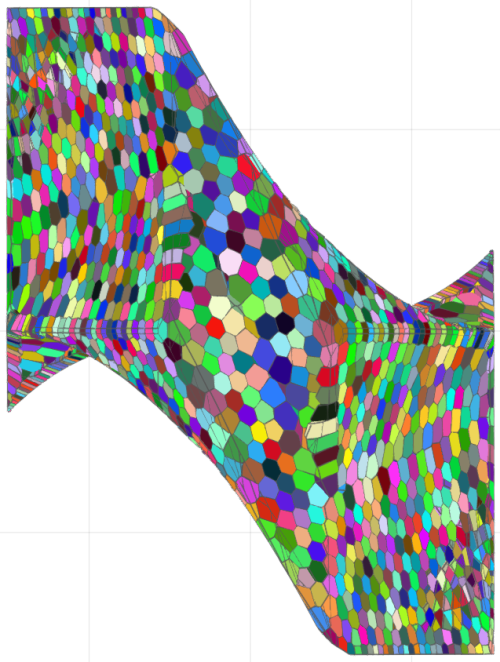}\quad
\includegraphics[width=0.45\textwidth]{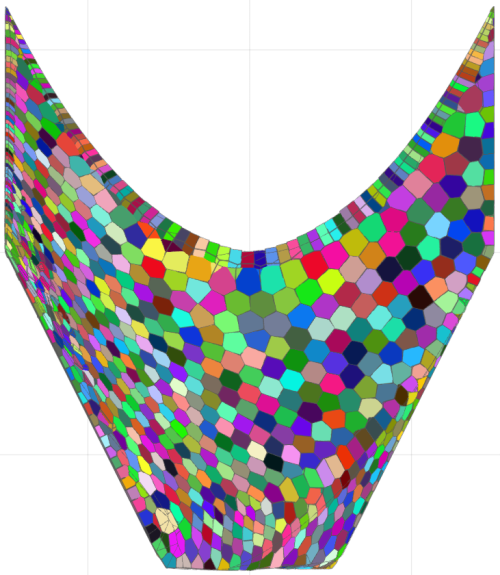}
\caption{Different views of the three dimensional Blaschke-Santal\'o diagram for $(\tr, \lambda_1\lambda_2+\lambda_2\lambda_3+\lambda_3\lambda_1,\det)$ in $\Sym_3$. %: 3D view, $Oyz$, $Oxz$, $Oxy$. 
General 3D view (top left), $Oyz$ view (top-right), $Oxz$ view (bottom-left), $Oxy$ view (bottom-right).}
\label{fig:3D-computations}.
\end{figure}

\subsection{Comparison with Monte Carlo method} The simplest approach to investigate the Blaschke-Santal\'o diagrams is generating random samples and computing the corresponding images. As underlined in Section \ref{sec:optimal-transport}, this choice does not necessarily produce images uniformly distributed in the desired diagram. In the following, we generate progressively, a fixed number of sample points and the corresponding images. We compare the quality of the result and the computational cost with the algorithms proposed in the previous sections.

We consider $10^4$ and, respectively $10^6$ random samples in $\Sym_d([-1,1])$, evaluate \eqref{eq:trace-det} and plot the corresponding points in $\Bbb{R}^2$. The corresponding results for $d \in \{2,3,4\}$ are shown in Figure \ref{fig:random}. The Blaschke-Santal\'o diagrams computed with the algorithms proposed in previous Sections are represented as polygons while random samples are represented by points. The diagrams are rescaled to have the same width in the horizontal and vertical directions.

We notice that the Monte Carlo approach is inefficient, especially when the dimension increases. For $d \in \{3,4\}$ using one million random samples fails to give an accurate description of the diagram.

\begin{figure}
\centering
\includegraphics[height=0.3\textwidth]{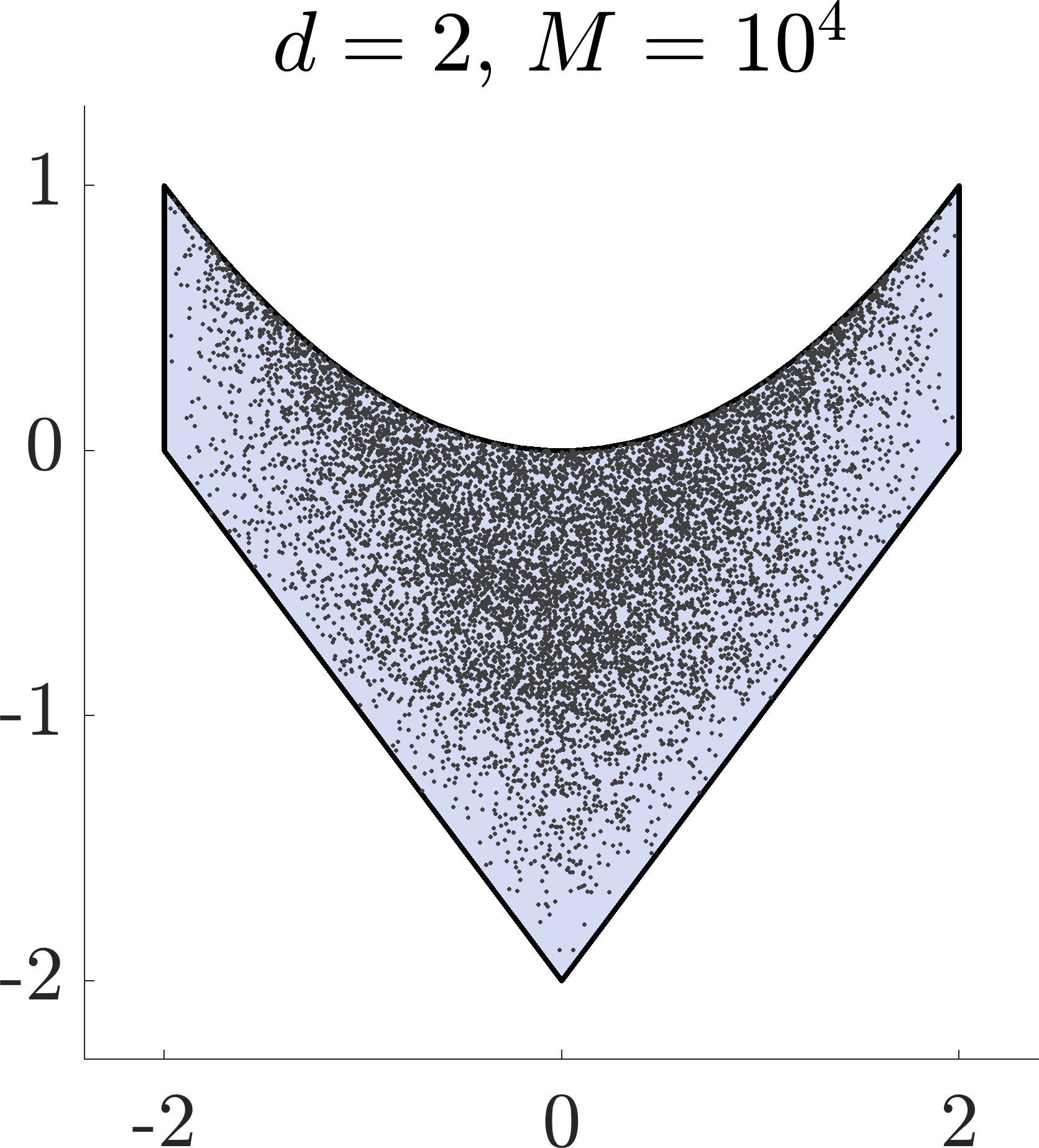}\quad
\includegraphics[height=0.3\textwidth]{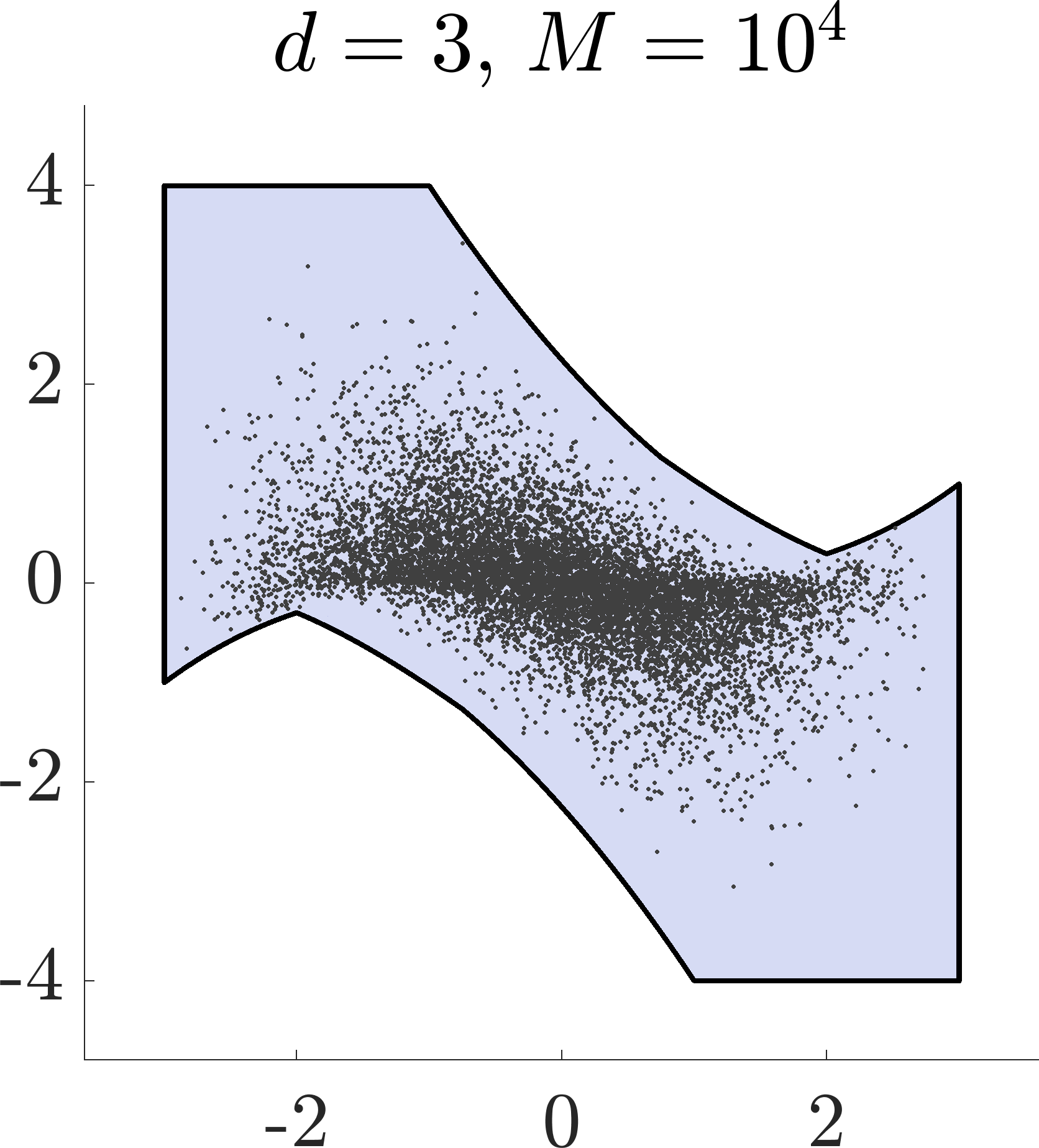}\quad
\includegraphics[height=0.3\textwidth]{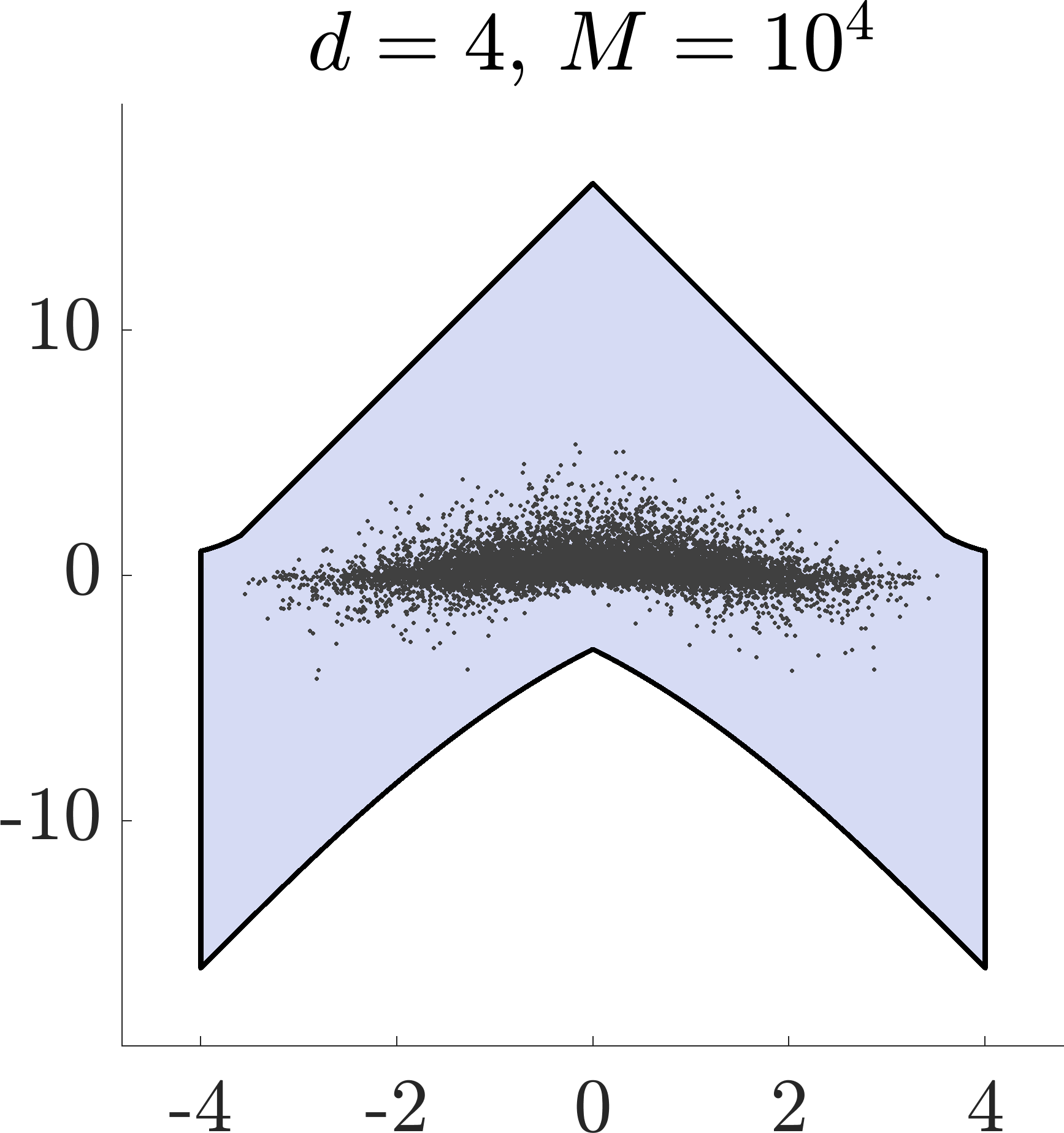} 

\vspace{0.3cm}
	
\includegraphics[height=0.3\textwidth]{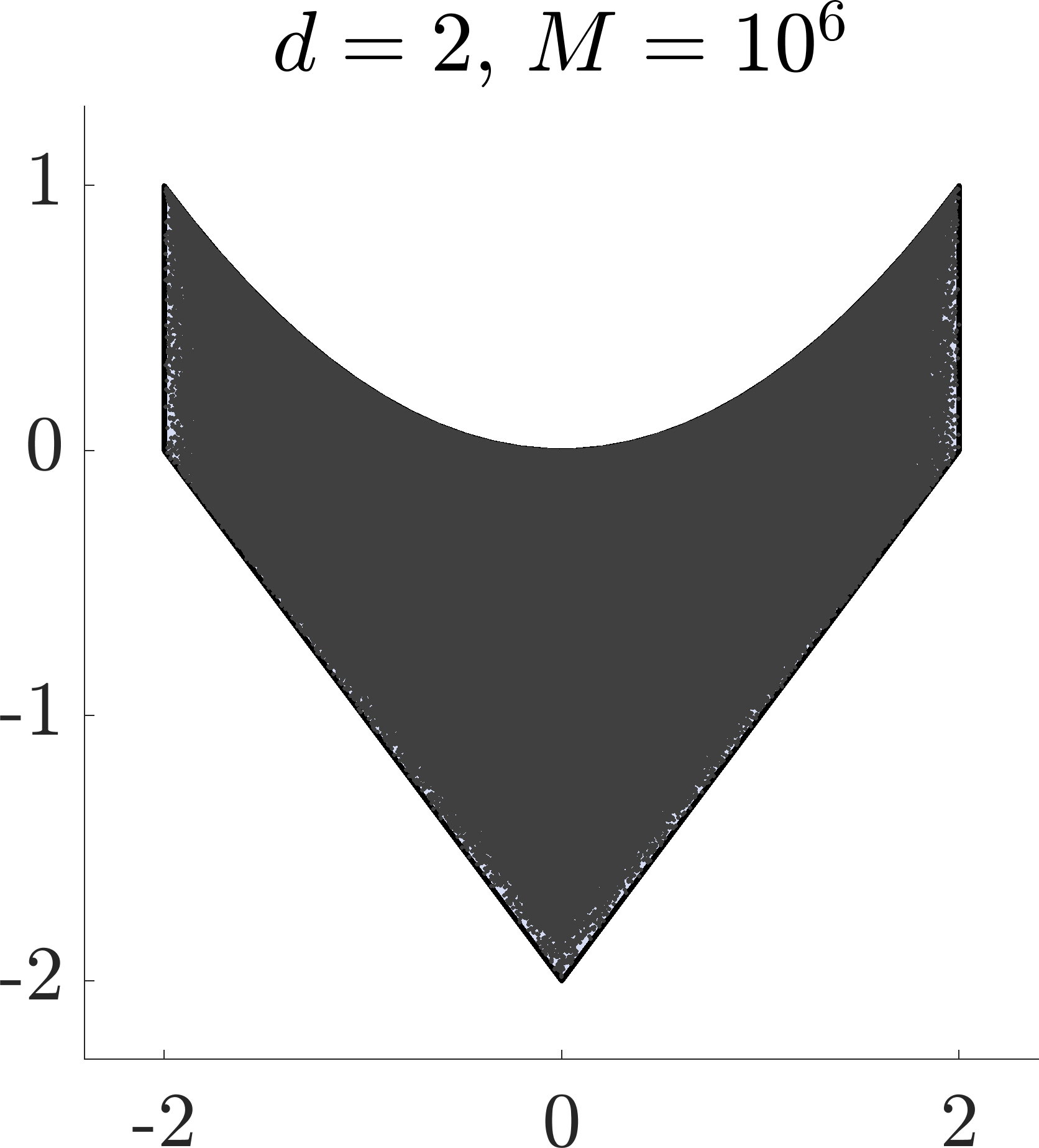}\quad
\includegraphics[height=0.3\textwidth]{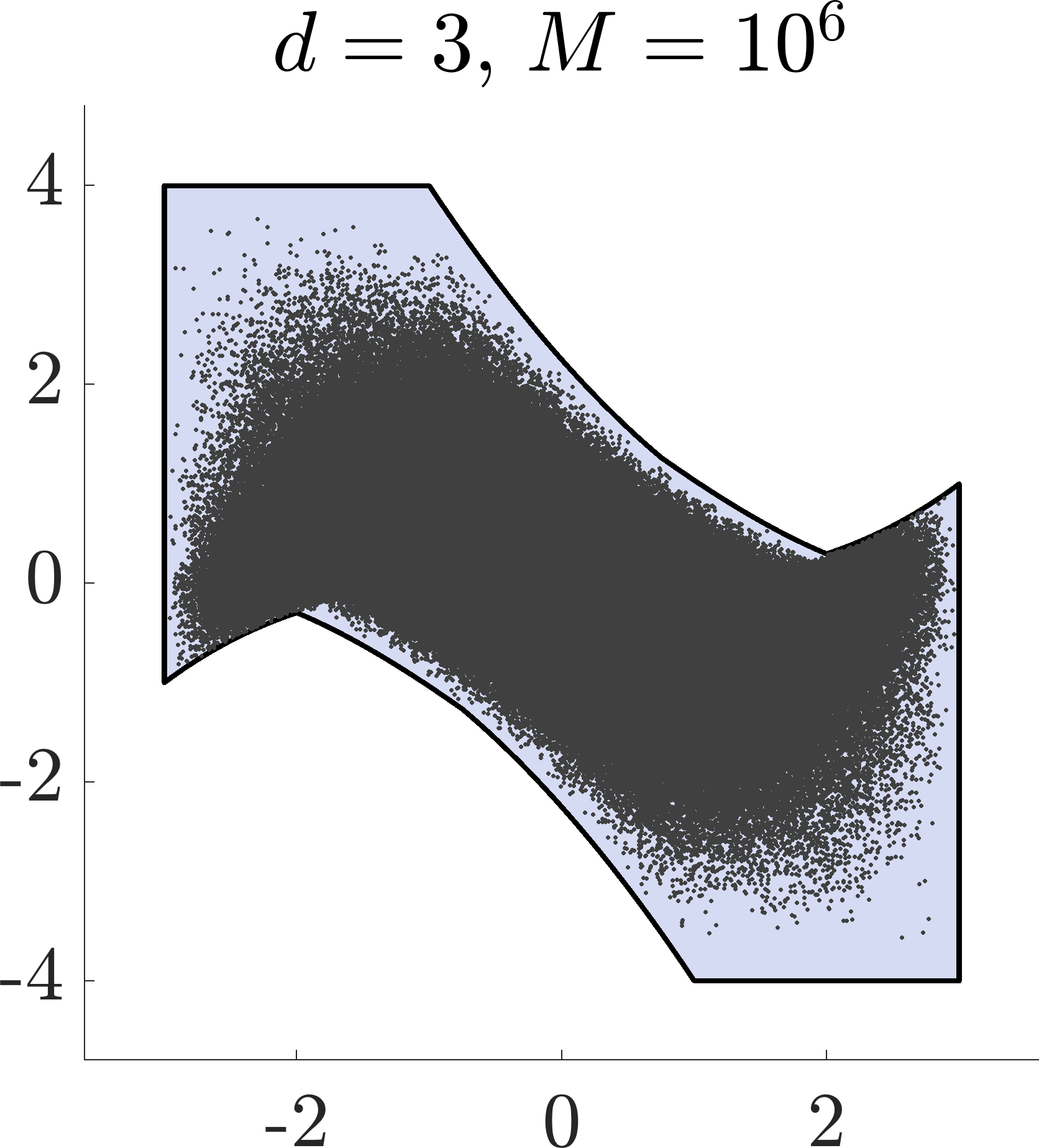}\quad
\includegraphics[height=0.3\textwidth]{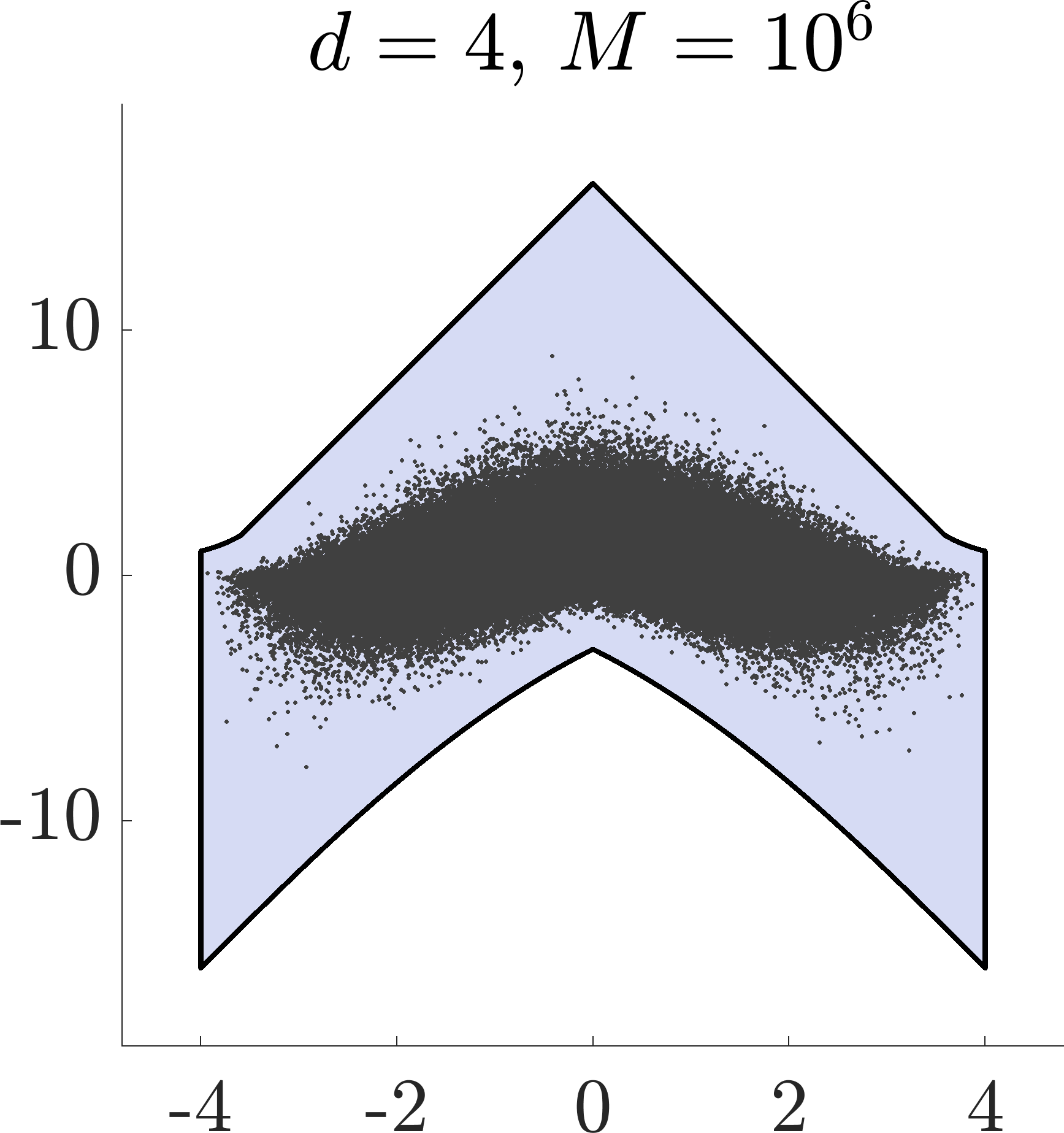}
\caption{Generating $10^4$ and, respectively, $10^6$ random matrices in $\Sym_d([-1,1])$ for cases $d\in\{2,3,4\}$ and plotting the corresponding images given by the function \eqref{eq:trace-det}.}
\label{fig:random}
\end{figure}

In comparison, we give an analysis of the computational cost for some of our simulations which give a high quality approximation of the Blaschke-Santal\'o diagrams. Simulations in Figures \ref{fig:single-lloyd} and \ref{fig:single-cvt} use $M=200$ samples with a limit of $1000$ iterations. For Algorithm \ref{algo:var-cvt} at most $200\times 1000$ function and gradient evaluations are preformed. The computation of the Voronoi diagrams using \texttt{Geogram} is very efficient. Algorithm \ref{algo:lloyd} performs additional function evaluations when projecting the centroids on the space of samples, but remains of the same order of magnitude: $O(M\times Q)$, where $Q$ is the number of iterations. 

%%%%%%%%%%%%%%%%%%%%%%%%%%%%%%%%%%%%%%%%%%%%%%%%%%
\section{Application II: example from convex geometry}

We focus now on an application from convex geometry. Various other works investigate inequalities between geometric quantities using Blaschke-Santal\'o diagrams. Among these we mention \cite{missingADR}, \cite{ftouhi-lamboley}, \cite{ftouhi-Cheeger}, \cite{ftouhi-numerics}, \cite{ftouhi-polya}. In order to apply directly our computational framework we consider a particular case where functionals involved are smooth and the corresponding diagram is bounded.

Consider the following three quantities: area $A(\Omega)$, perimeter $\Per(\Omega)$, momentum of inertia $W(\Omega)$ among two dimensional convex shapes $\Omega$ with two axes of symmetry. Since in this case the centroid is at the origin, the momentum of inertia is given by $W = \int_\Omega |x|^2 dx$.

As usual, when studying Blaschke-Santal\'o diagrams, we consider scale invariant quantities linking the three functionals. One can naturally consider:
\begin{itemize}
\item the isoperimetric ratio $A(\Omega)/\Per(\Omega)^2$, bounded above by $\frac{1}{4\pi}$. 
\item the ratio $A^2(\Omega)/W(\Omega)$, bounded above by $2\pi$.
\end{itemize}
One can notice that both scale invariant ratios considered above are maximized by the disk. Theoretical details regarding the corresponding Blaschke-Santal\'o diagram are studied in \cite{APW-preprint}.

We consider the mapping 
\begin{equation}
\F:\Omega\mapsto (100 A(\Omega)/P^2(\Omega),A^2(\Omega)/W(\Omega)),
\label{eq:APW}
\end{equation} where the factor $100$ is added so that the two quantities are comparable. Our objective is to approximate the image of the mapping $\F$ defined above. 

Various methods were developed for parametrizing convex sets. We mention intersections of hyperplanes \cite{LROconvex}, the support function parametrized using truncated Fourier series in \cite{AntunesBogosel} or values on a discrete grid in \cite{DiscreteConvex}. Methods proposed previously are generally based on linear inequality constraints on the set of parameters. In order to apply our framework directly, a more direct parametrization, using only bound constraints would be more appropriate. This leads us to propose an alternate, yet classical, discretization process. 

We focus on the class of convex sets with two axes of symmetry. Since we are also working in a scale invariant setting, it is enough to parametrize concave and decreasing functions $y:[0,1]\to \Bbb{R}$. Given a uniform discretization of $[0,1]$ using $q+1$ points, observe that if $y_0,...,y_q$ are samples of a concave decreasing function at $x_i = i/q, i=0,...,q$ then:
\begin{itemize}
	\item the first order differences are $z_i=y_i-y_{i+1}, i=0,...,q-1$ are increasing
	\item the second order differences $\rho_{i+1}=z_{i+1}-z_{i}, i=0,...,q-2$ are non-negative.
\end{itemize}
Conversely, given non-negative values $\rho_i$, it is possible to construct samples of a concave decreasing function having $\rho_i$ as second order differences. Therefore we take $(\rho_i)_{i=1}^{q-1}$, $\rho_0:=z_0$ and $y_q$ as variables in our parametrization.

%$\star$ Everything can be computed smoothly in terms of the parameters.
\begin{figure}
	\includegraphics[width=0.6\linewidth]{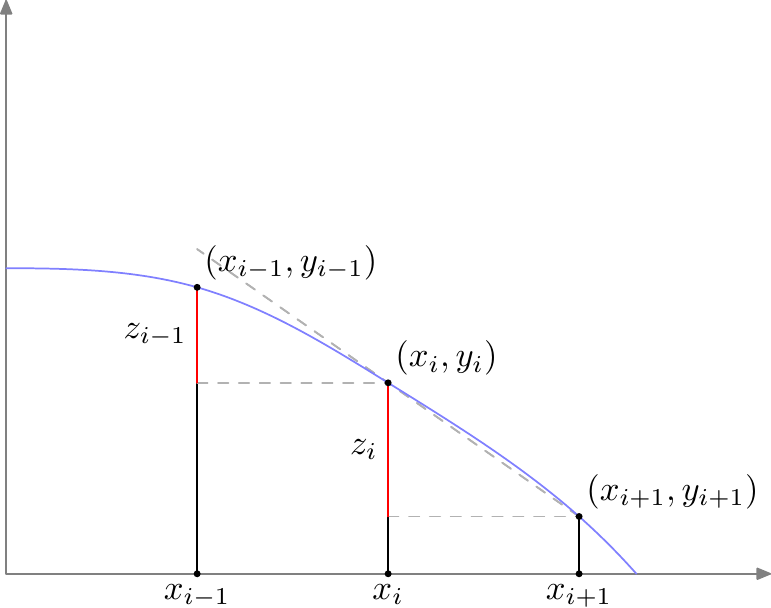}
	\caption{Parametrization of concave decreasing function using second order differences. The first order differences $z_i$ are decreasing.}
	\label{fig:convex-param}
\end{figure}

We immediately obtain the following equalities
\[y_i=y_q+\sum_{j=i}^{q-1} z_i,\qquad z_i=\sum_{k=0}^{i-1}\rho_k,\qquad i=0,...,q-1,\]
which show that $y_i$, for $i=0,...,q-1$, can be expressed in terms of $y_q$ and $(\rho_i)_{i=0}^{q-2}$ using the following expression:
\begin{equation}
\begin{pmatrix}
y_0 \\ y_1 \\...\\ y_{q-1}
\end{pmatrix}
=y_q+A\begin{pmatrix}\rho_0 \\ \rho_1 \\...\\ \rho_{q-1}
\end{pmatrix}
\label{eq:explicit-param}
\end{equation}
with $A=(A_{ij})$ given by $A_{ij} = q+1-\max\{i,j\}$ for $1 \leq i,j \leq q$. The coordinates of the boundary points of the discrete convex set are given by $(x_i,y_i), i=0,...,q$ for the first quadrant. They are symmetrized to obtain the rest of the boundary. The area and the perimeter are computed in a straightforward way. For the momentum of inertia, we use the explicit formulas for polygons, found for example in \cite{Soerjadi1968OnTC}, a direct consequence of Green's formulas. Since all computations are analytic in terms of the parameters, the partial derivatives of all quantities of interest are also computed analytically. 

{\bf (a) Randomly generated shapes.} Given the parametrization above and a number of parameters $q\geq 2$ we can generate random convex shapes and plot the points given by \eqref{eq:APW}. We generate $1000$ random shapes for $2\leq q \leq 10$ parameters. The results are plotted in Figure \ref{fig:randomAPW}. It can be observed that $q=2$ produces points on the upper part of the boundary, while higher values of $q$ produce points closer to the origin. In fact, as $q$ increases, the random shapes give points concentrated around the origin $(0,0)$. 

\begin{figure}
\centering
\includegraphics[width=0.6\textwidth]{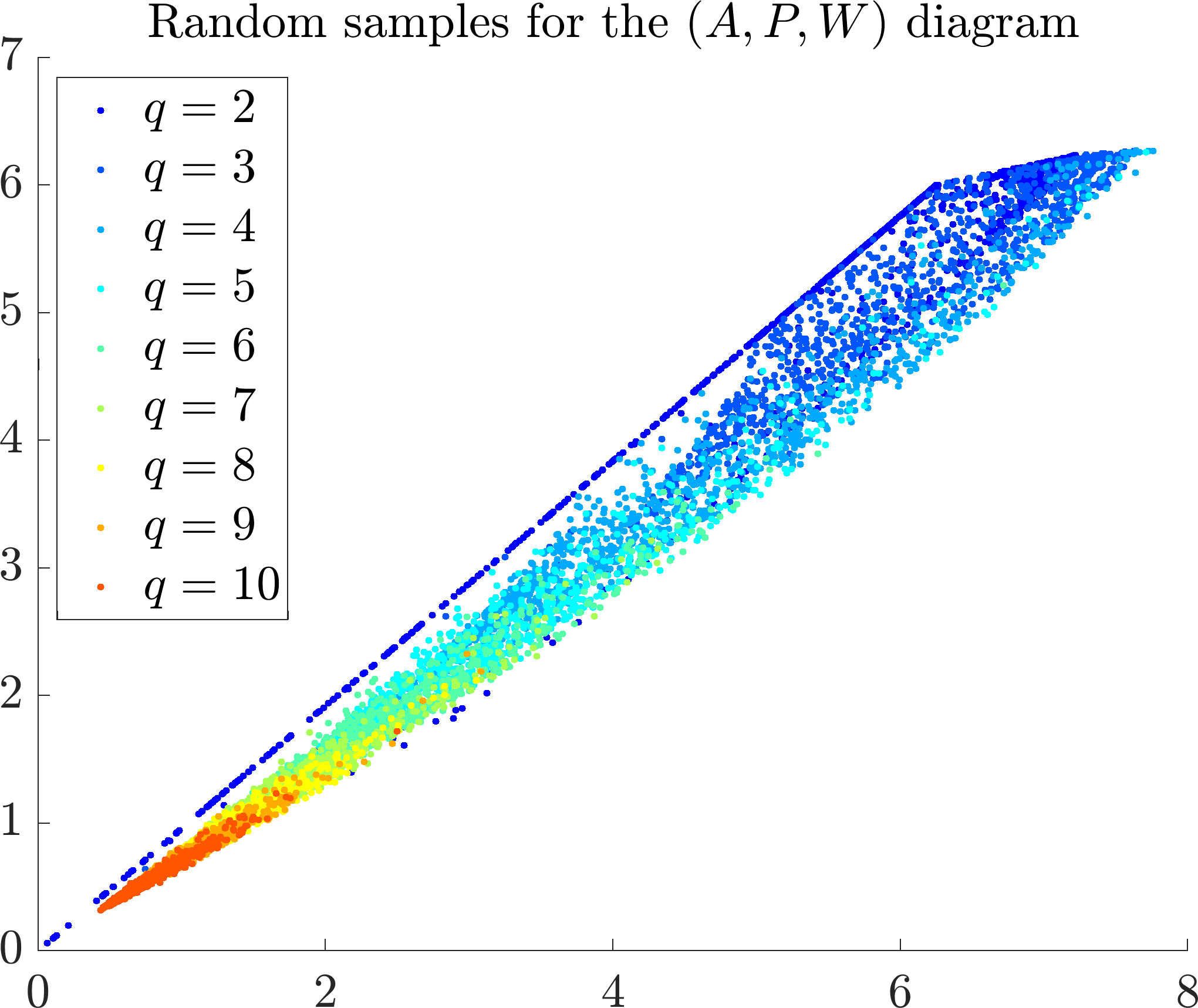}
\caption{Random sampling for the $(A,P,W)$ diagram using $q$ parameters for generating convex shapes, $q \in \{2,3,...,10\}$.}
\label{fig:randomAPW}
\end{figure}

{\bf (b) Using the numerical algorithms proposed in Section \ref{sec:framework}.}

We choose to work with $50$ parameters in \eqref{eq:explicit-param} for generating convex shapes. We generate $15$ random samples obtaining points very close to the origin, shown in the first image in Figure \ref{fig:APW}. The initial points do not give any meaningful information on the geometry of the diagram. However, applying Algorithm \ref{algo:global-ref} distributes these initial samples uniformly as shown in the second image in the same Figure. Then we continue the process, using the midpoints of edges of the Delaunay triangulation for adding more samples to the diagram. The multi-grid strategy uses $15$, $44$, $145$, $516$ samples, respectively. The final configuration uses $516\times 50 =25800$ parameters for the global iterative process.
\begin{figure}
\centering 
\includegraphics[width=0.3\textwidth]{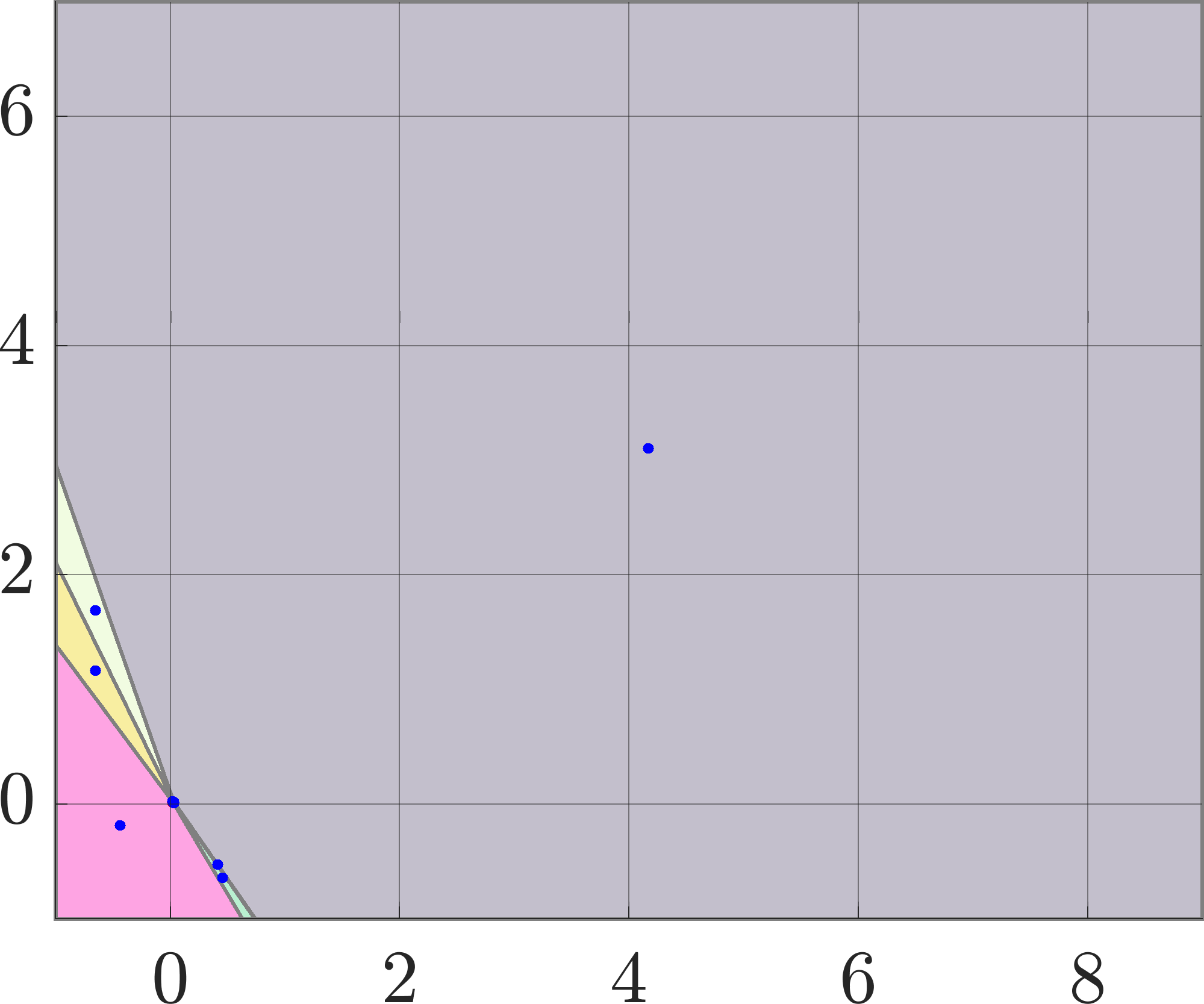}\quad
\includegraphics[width=0.3\textwidth]{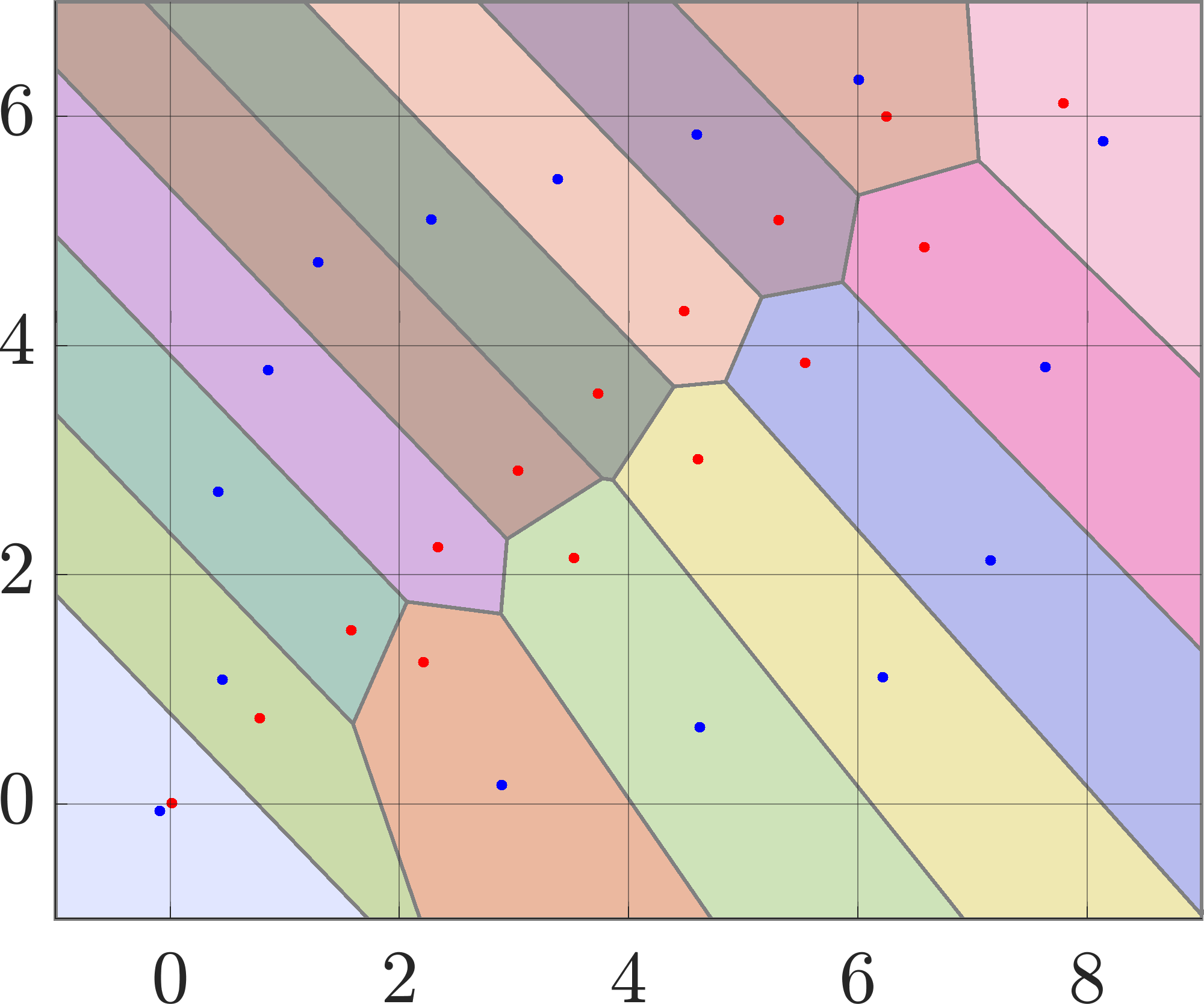}\quad
\includegraphics[width=0.3\textwidth]{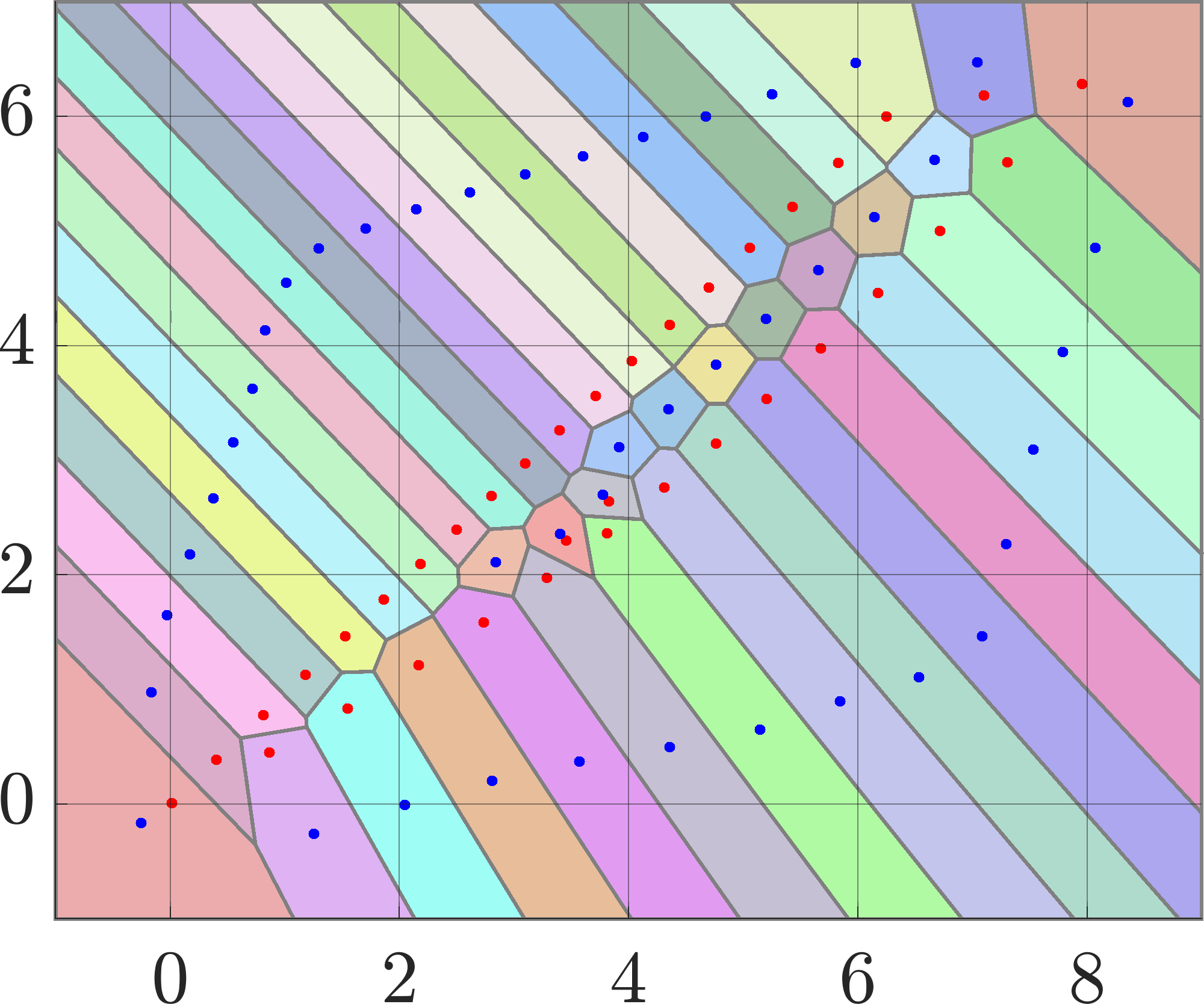}

\includegraphics[width=0.45\textwidth]{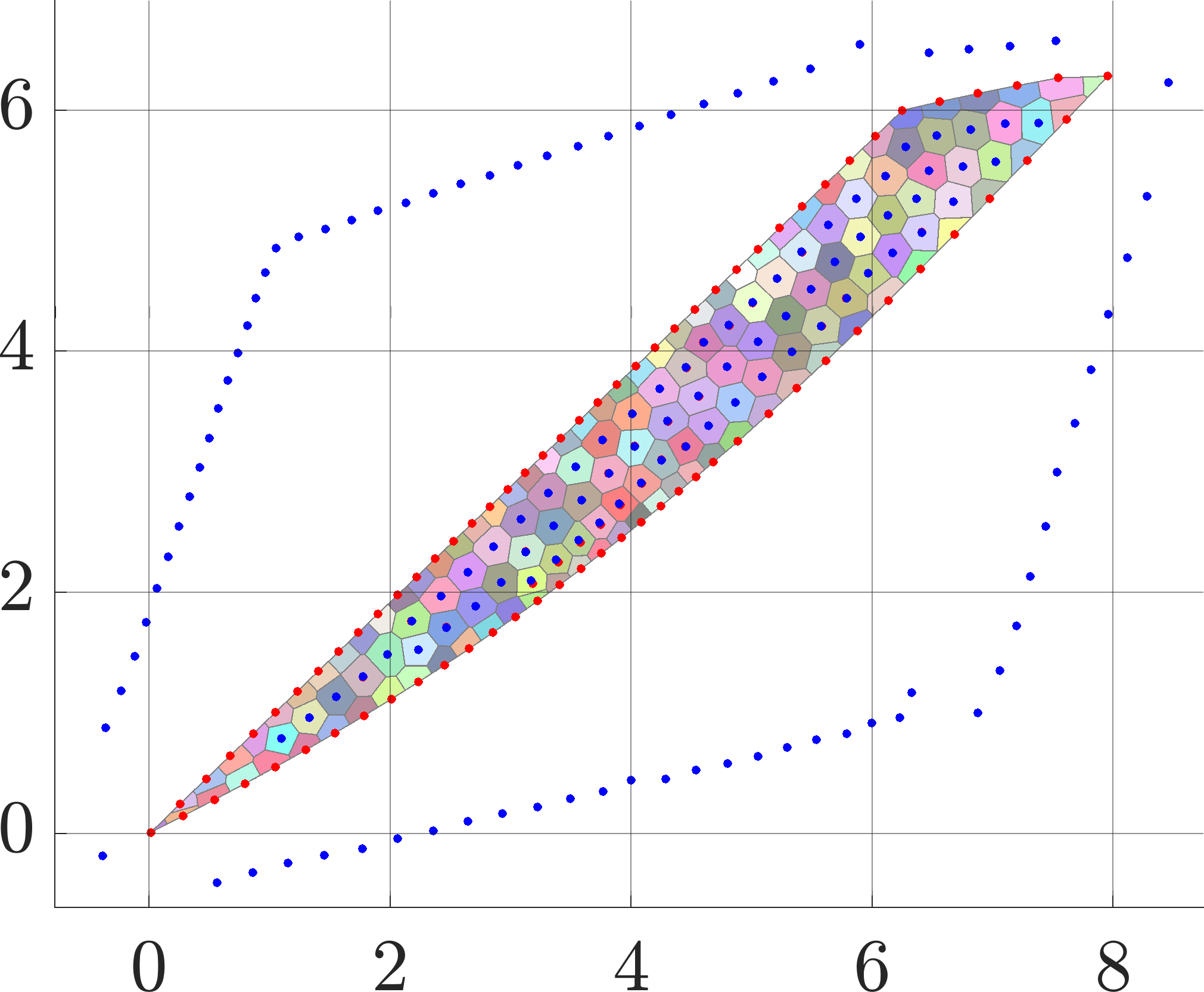}\quad
\includegraphics[width=0.45\textwidth]{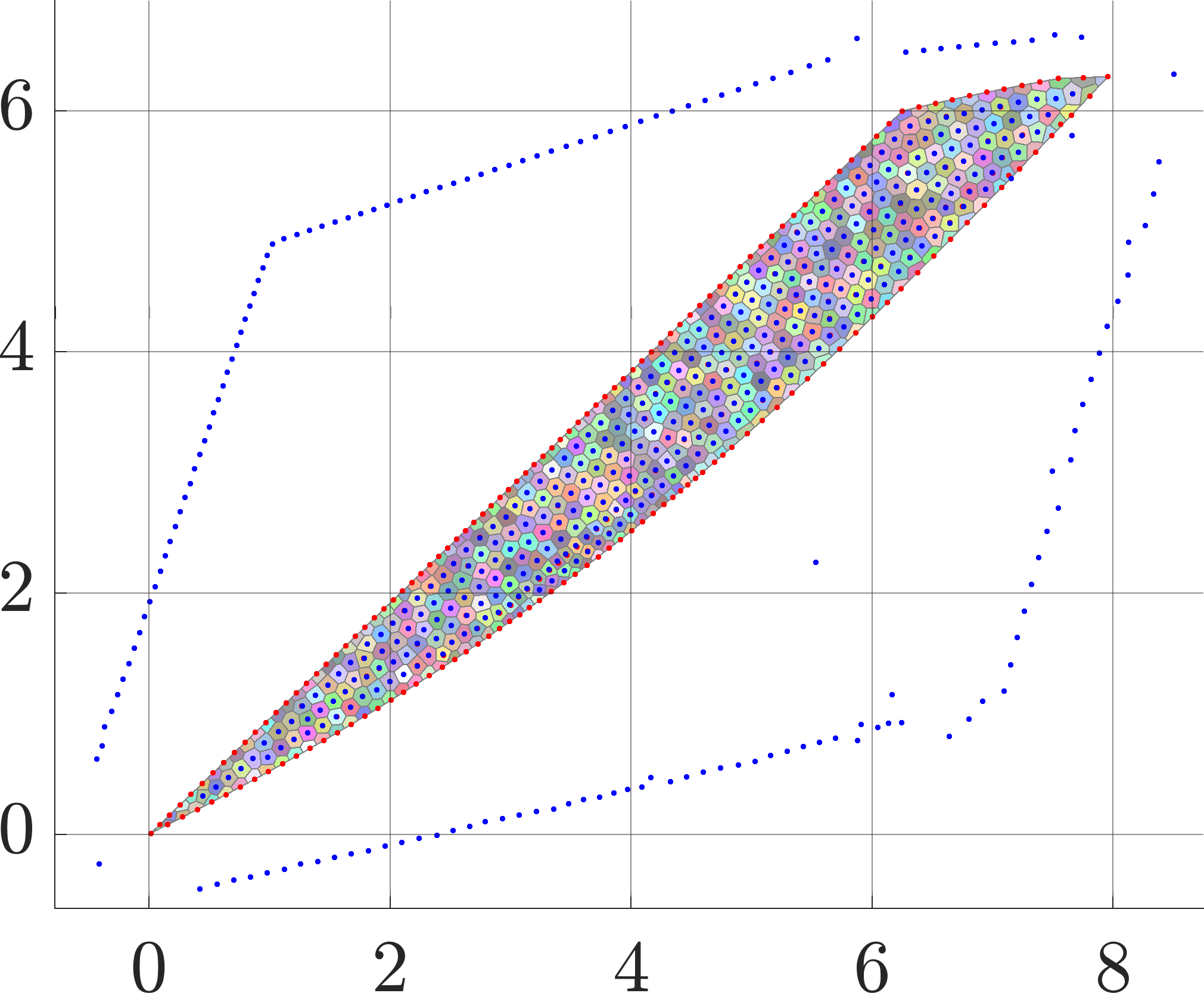}
\caption{Numerical approximation of the (area, perimeter, momentum of inertia) diagram.}
\label{fig:APW}
\end{figure}

Investigating shapes lying on the boundary of the Blaschke-Santal\'o diagram, shown in Figure \ref{fig:APW-Detail}, we observe the following:
\begin{itemize}
	\item The left upper boundary is generated by rhombi, flat towards the origin, going towards the square.
	\item The right upper boundary is generated by octagons and other polygons converging to the disk, corresponding to the upper-right corner of the diagram.
	\item The lower boundary contains shapes similar to stadiums or ellipses.
\end{itemize}

\begin{figure}
	\centering
	\includegraphics[width=0.8\textwidth]{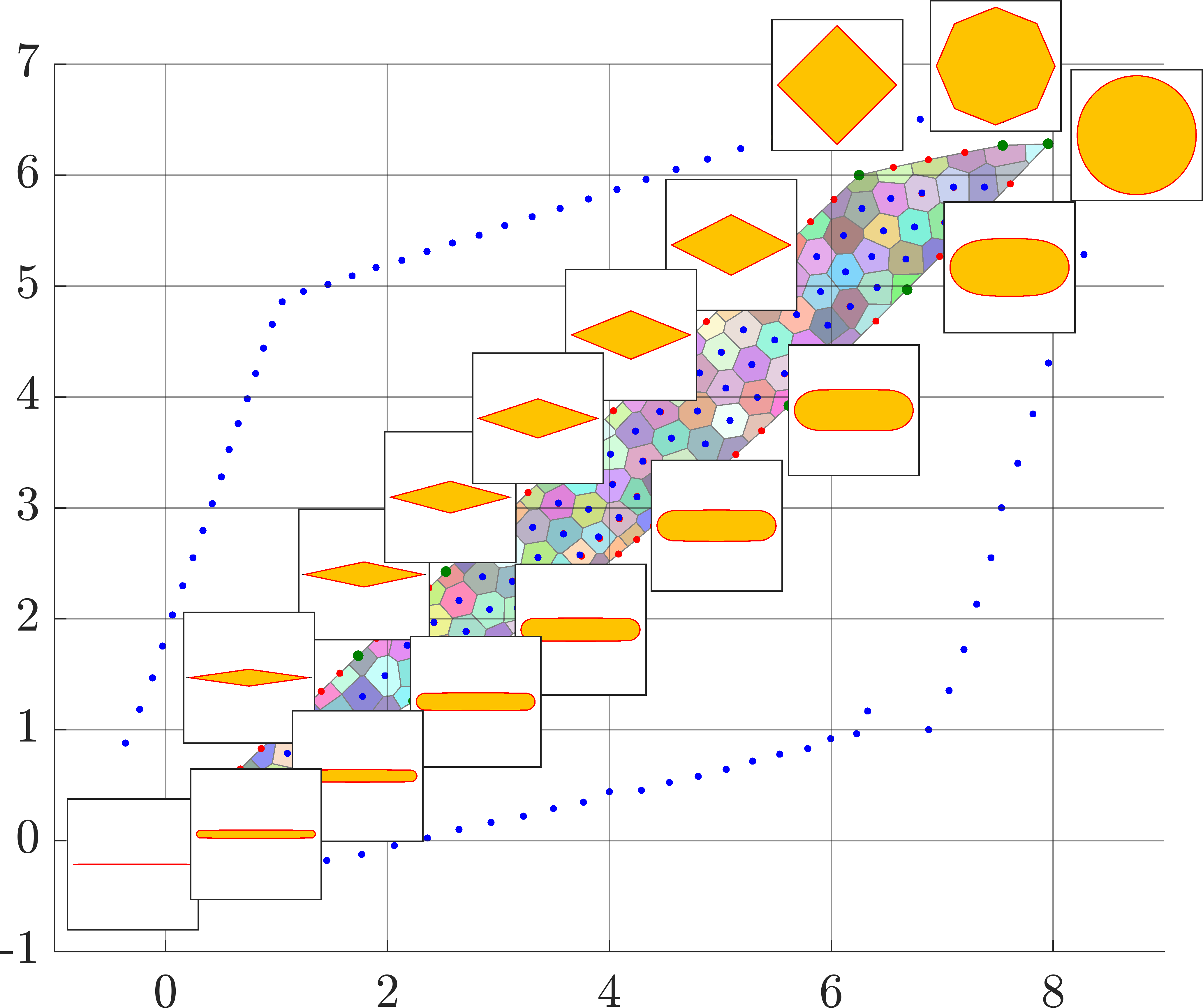}
	\caption{Examples of shapes corresponding to points on the boundary of the (Area, Perimeter, Momentum of inertia) diagram.}
	\label{fig:APW-Detail}
\end{figure}

The algorithms proposed, based on Centroidal Voronoi Tessellations, give a really accurate numerical description of the diagram. Like in the algebraic case, investigating the shapes on the boundary of the diagram provides insights regarding the possible analytical bounds and may guide theoretical study to obtain a complete description. More details regarding this diagram are given in \cite{APW-preprint}.

%%%%%%%%%%%%%%%%%%%%%%%%%%%%%%%%%%%%%%%%%%%%%%%%%%
\section{Conclusions}

We propose efficient algorithms which approximate Blaschke-Santal\'o diagrams by generating samples having uniformly distributed images. The key ingredient is the search for images which produce Centroidal Voronoi Tessellations. The algorithms proposed, inspired from Lloyd's algorithm and the Variational method proposed in \cite{cvt-levy} are illustrated through examples coming from linear algebra and shape optimization. 

We observe that using a reasonable computational cost, compared with the usual Monte Carlo methods which generate randomized samples, the algorithms proposed achieve a precise description of the Blaschke-Santal\'o diagrams. Using a multi-grid strategy, more samples can be considered, further improving the description of these diagrams.

%%%%%%%%%%%%%%%%%%%%%%%%%%%%%%%%%%%%%%%%%%%%%%%%%%

\bigskip

\noindent{\bf Acknowledgments.} The work of GB is part of the project 2017TEXA3H {\it``Gradient flows, Optimal Transport and Metric Measure Structures''} funded by the Italian Ministry of Research and University. The author is member of the Gruppo Nazionale per l'Analisi Matematica, la Probabilit\`a e le loro Applicazioni (GNAMPA) of the Istituto Nazionale di Alta Matematica (INdAM). The work of BB and EO was supported by the ANR Shapo (ANR-18-CE40-0013) program. The authors of \cite{APW-preprint} are warmly thanked for sharing with us their preliminary results.

\bigskip

%\bibliographystyle{abbrv}
%\bibliography{./biblio.bib}

\bigskip
\small\noindent
Beniamin Bogosel: Centre de Math\'ematiques Appliqu\'ees, CNRS,\\
\'Ecole polytechnique, Institut Polytechnique de Paris,\\
91120 Palaiseau, France \\
{\tt beniamin.bogosel@polytechnique.edu}\\
{\tt \nolinkurl{http://www.cmap.polytechnique.fr/~beniamin.bogosel/}}

\bigskip
\small\noindent
Giuseppe Buttazzo:
Dipartimento di Matematica,
Universit\`a di Pisa\\
Largo B. Pontecorvo 5,
56127 Pisa - ITALY\\
{\tt giuseppe.buttazzo@dm.unipi.it}\\
{\tt http://www.dm.unipi.it/pages/buttazzo/}

\bigskip
\small\noindent
Edouard Oudet:
Laboratoire Jean Kuntzmann (LJK),
Universit\'e Joseph Fourier\\
Tour IRMA, BP 53, 51 rue des Math\'ematiques,
38041 Grenoble Cedex 9 - FRANCE\\
{\tt edouard.oudet@imag.fr}\\
{\tt http://www-ljk.imag.fr/membres/Edouard.Oudet/} 

\end{document}